\documentclass[a4paper,12pt]{amsart}
\usepackage{amsfonts,amsthm,amssymb,amsmath,amscd}
\usepackage[colorlinks=true,linkcolor=blue,citecolor=blue]{hyperref}
\usepackage{amsmath}
\usepackage{tikz-cd}

\usepackage{tabularx}
\usepackage{tcolorbox}
\usepackage[english]{babel}
\usepackage{amssymb,amsmath}
\usepackage{xcolor}
\usepackage{enumerate}

\newtheorem{Theo}{Theorem}[section]
\newtheorem{Prop}[Theo]{Proposition}
\newtheorem{Coro}[Theo]{Corollary}
\newtheorem{Lemm}[Theo]{Lemma}

\newtheorem{Rema}[Theo]{Remark}

\newcommand{\Hcal}{\mathcal{H}}
\newcommand{\T}{\mathbb{T}}
\newcommand{\Bcal}{\mathcal{B}}
\newcommand{\D}{\mathbb{D}}
\newcommand{\C}{\mathbb{C}}
\newcommand{\Z}{\mathbb{Z}}

\DeclareMathOperator{\spa}{span}

\def\N{\mathbb{ N}}
\def\R{\mathbb{ R}}

\def\HH{\mathcal{ H}}

\DeclareMathOperator{\rad}{Rad}
\DeclareMathOperator{\Rad}{RAD}
\def\Q{\mathbb{ Q}}

\begin{document}

\title{Vector-valued general Dirichlet series }

\author{
D.~Carando, A.~Defant, F.~Marceca, I.~Schoolmann}

%
%


\begin{abstract}
Opened up by  early contributions due to, among others, H. Bohr, Hardy-Riesz, Bohnenblust-Hille, Neder and Landau
the last 20 years show a substantial revival of  systematic research on ordinary Dirichlet series $\sum a_n n^{-s}$, and more recently even on general Dirichlet series
$\sum a_n e^{-\lambda_n s}$.  This involves the intertwining of  classical work with modern functional analysis, harmonic analysis, infinite dimensional holomorphy and probability theory as well as analytic number theory.
Motivated through this line of research the main goal of this article is to start a systematic study of a variety of fundamental aspects of vector-valued general Dirichlet series
$\sum a_n e^{-\lambda_{n} s}$, so Dirichlet series,   where the coefficients are not necessarily in $\C$ but in some arbitrary Banach space $X$.
\end{abstract}

\maketitle


\tableofcontents

\noindent
\renewcommand{\thefootnote}{\fnsymbol{footnote}}
 \footnotetext{}

\section{Introduction}

Given a frequency $\lambda = (\lambda_n)$, i.e. a strictly increasing sequence of non-negative real numbers,
a $\lambda$-Dirichlet series  is a (formal) series of the form $D=\sum a_{n}e^{-\lambda_{n}s}$, where $s$ is a complex variable and the sequence  $(a_{n})$
(of so-called Dirichlet coefficients) belongs to $\C$.

In contrast to the theory of general Dirichlet series $D=\sum a_{n}e^{-\lambda_{n}s}$, the theory of  ordinary Dirichlet series $\sum a_{n} n^{-s}$ saw a remarkable renaissance
within the last two decades
  which in particular led to the solution of some long-standing problems.
One of many fruitful lines of research in this respect is given by the analysis of functional analytic aspects of vector-valued ordinary Dirichlet series, so  series $\sum a_{n} n^{-s}$
with coefficients $a_n$ in a given normed  space $X$. The list
\cite{CarandoDefantSevilla},
\cite{CarandoDefantSevilla1},
\cite{CarandoMaceraScottiTradacete},
\cite{CarandoMarcecaSevilla},
\cite{CastilloMedinaGarciaMaestre},
\cite{DefantGarciaMaestrePerez},
\cite{DefantPopaSchwarting}, \cite{DefantSchwartingSevilla},
\cite{DefantSevilla} of recent articles indeed documents this activity; let us also mention that  some of the results proved in these articles are   collected in the recent monograph \cite[Chapter 26]{Defant}.

Motivated through this line of research the main goal of this article is to start a systematic study of a variety of aspects of vector-valued general Dirichlet series $D=\sum a_{n}e^{-\lambda_{n}s}$ with coefficients $a_n$ in a given Banach space $X$.
Two different challenges combine in this setting: the behaviour of vector-valued general Dirichlet series depends not only on the structure of the frequency $\lambda$ but also on the geometric structure of the normed space $X$.

Regarding the frequency, many important tools developed within the study of ordinary Dirichlet series (un)fortunately  fail  for general Dirichlet series. In other words,   making the jump from the frequency
$(\log n)$ to arbitrary frequencies $\lambda$ reveals challenging consequences. Much of the ordinary theory relies on 'Bohr's theorem', a fact which in particular implies that for each ordinary Dirichlet series the abscissas of uniform convergence and boundedness coincide. However for general Dirichlet series, the  validity of Bohr's theorem  depends very much on the 'structure' of the frequency $\lambda$.

Further due to the fundamental theorem of arithmetic, each natural number $n$ has its prime number decomposition $n=\mathfrak{p}^{\alpha}$, where $\alpha\in\N_0^{(\N)}$ and $\mathfrak{p} = (2,3, \ldots)$ stands for the sequence of primes, and so  the frequency $(\log n)$ can be written as a linear combination of $(\log p_{j})$ with natural coefficients. This allows to translate Dirichlet series $\sum a_{n} n^{-s}$ into power series of infinitely many variables $\sum_\alpha a_{p^\alpha}z_1^{\alpha_1}z_2^{\alpha_2}\ldots$ in a natural way providing an intimate link between the theory of ordinary Dirichlet series  and the theory of holomorphic functions and polynomials on polydiscs.  Known as Bohr's transform, this procedure enables powerful tools to enter the game. However, it is no longer entirely available for arbitrary frequencies.

Regarding vector-valued phenomena, whereas some scalar-valued results can be translated directly to the vector-valued setting, others depend on the geometric structure of the normed space. Furthermore, some geometric restrictions vary depending on the frequency $\lambda$.

To illustrate all this, we in the rest of this introduction describe  four aspects of ordinary as well as general Dirichlet series which are going to guide us in establishing a systematic theory of general Dirichlet series with coefficients in normed spaces. The Sections \ref{sec3vectorvalued}, \ref{sec4vectorvalued} and \ref{maximalinequalitiessecvectorvalued} are devoted to each of these aspects -- but  in Section~\ref{sec2vectorvalued} we first start with some more preliminaries followed by the definition of  several important classes of general Dirichlet series.

\bigskip

\noindent{\bf Aspect I.}
Let $\mathcal{D}_{\infty}(\lambda,X)$ be the space of all $\lambda$-Dirichlet series $D:=\sum a_{n} e^{-\lambda_n s}$ with coefficients in $X$ which converge and define a bounded, and then necessarily holomorphic, function on the open right half plane $[Re>0]$
(endowed with the supremum norm on $[Re>0]$).

   The countable product  $\mathbb{T}^\infty$ of the torus $\T$  forms a natural compact abelian group,
 where  the Haar measure is given by the  normalized Lebesgue measure.
Recall the definition of the Hardy space $H_\infty(\mathbb{T}^\infty,X)$, i.e. the closed subspace of all $f \in L_\infty(\mathbb{T}^\infty,X)$
such that $\widehat{f}(\alpha) =0$ if $\alpha\in \Z^{(\N)}\setminus \N_0^{(\N)}$.
Using Bohr's transform, these Banach spaces provide us with a  notion of 'Hardy space' for ordinary Dirichlet series which we denote by $\mathcal{H}_\infty ((\log n),X)$ (see Section~\ref{HplambdaX} for the precise definition); whenever $X = \C$ we write $\HH_{\infty}((\log n))$.

One of the celebrated results in the  scalar ordinary theory is a result of Hedenmalm, Lindqvist and Seip from  \cite{HLS}  which shows
\begin{equation} \label{HLSvectorvalued}
\mathcal{D}_{\infty}((\log n)) = \HH_{\infty}((\log n)).
\end{equation}
  This equality can be proven via the intermediate space $\HH^+_{\infty}((\log n))$. To understand this, define $\HH^+_{\infty}((\log n),X)$ as the Banach space of all $X$-valued ordinary Dirichlet series $D=\sum a_{n} n^{-s}$ such that all its  translates
  $D_\sigma:=\sum a_{n} n^{-\sigma} n^{-s}, \sigma >0,$ form a  uniformly bounded set of $\HH^+_{\infty}((\log n),X)$, together with the norm
\[\|D\|_\infty^+:=\sup_{\sigma>0}\|D_\sigma\|_\infty.\]
As a particular case of ideas from \cite{AntonioDefant} (elaborated in \cite[Theorems 24.8 and 24.17]{Defant}) we have
\begin{align}\label{HLS2vectorvalued}
\HH_{\infty}((\log n))=\HH^+_{\infty}((\log n))=\mathcal{D}_{\infty}((\log n)).
\end{align}
Both identities are proven separately in order to show \eqref{HLSvectorvalued}, while the original proof from \cite{HLS} takes a detour through infinite dimensional holomorphy using a result of Cole and Gamelin from \cite{ColeGamelin} (see also \eqref{hol_pvectorvalued}).

As it turns out, studying vector-valued general Dirichlet series, provides the right setting to shed some light on why the proof is done in two steps.
Regarding vector-valued ordinary Dirichlet series, in \cite{AntonioDefant} (see also \cite[Theorem 24.8]{Defant}) it is actually shown that
\begin{align}\label{arnpvectorvalued}
\HH_{\infty}((\log n),X)=\HH^+_{\infty}((\log n),X),
\end{align}
if and only if the Banach space $X$ has the so-called analytic Radon-Nikodym property ARNP (see \cite[Chapter 23]{Defant} for a definition). On the other hand, by  \cite{AntonioDefant} (see also \cite[Theorem 24.17]{Defant})
\[\HH^+_{\infty}((\log n),X)=\mathcal{D}_{\infty}((\log n),X),\]
regardless of the geometry of $X$.

As shown in Section \ref{HplambdaX}, the definition of $\HH_{\infty}((\log n),X)$ and $\HH^+_{\infty}((\log n),X)$ can be extended to general frequencies to define $\HH_{\infty}(\lambda,X)$ and $\HH^+_{\infty}(\lambda,X)$. In fact, a standard weak compactness argument shows that for scalar generalized Dirichlet series we always have that
\begin{align}\label{motiv1vectorvalued}
\HH_{\infty}(\lambda)=\HH^+_{\infty}(\lambda)
\end{align}
(see also Proposition~\ref{Hpplusscalarvectorvalued}), whereas together with \cite[Theorem 5.1]{DefantSchoolmann4} the equality
\begin{align}\label{motiv2vectorvalued}
\HH^+_{\infty}(\lambda)=\mathcal{D}_{\infty}(\lambda)
\end{align}
holds if and only if $\lambda$ satisfies the Bohr's Theorem. Two questions arise.
\begin{enumerate}[I.]
	\item \label{q1vectorvalued} For which frequencies $\lambda$ and Banach spaces $X$ does $\HH_{\infty}(\lambda,X)=\HH^+_{\infty}(\lambda,X)$ hold?
	\item \label{q2vectorvalued} For which frequencies $\lambda$ and Banach spaces $X$ does $\HH^+_{\infty}(\lambda,X)=\mathcal{D}_{\infty}(\lambda,X)$ hold?
\end{enumerate}
The answer to these questions is at the core of Sections \ref{sec3vectorvalued} and \ref{sec4vectorvalued}, respectively.

\bigskip

\noindent{\bf Aspect II.}
 As mentioned before, ordinary Dirichlet series  can be identified with holomorphic functions on polydiscs. More precisely, denote by $H_\infty (B_{c_0}, X)$ the Banach space of all holomorphic (Fr\'echet differentiable) functions
 $g:B_{c_0}\rightarrow X$ endowed with the sup norm.
  There  is a unique isometric linear bijection
 \begin{equation} \label{hol_pvectorvalued}
 \mathcal{D}_{\infty}((\log n),X)= H_\infty (B_{c_0},X)\,,
 \end{equation}
 which preserves Dirichlet and monomial coefficients (see \cite{ColeGamelin} for the scalar case, \cite{AntonioDefant} for the vector-valued case, and also \cite[Theorem 24.17]{Defant}).

 Regarding $1 \leq p < \infty$, define the Banach space $H_{p}(\ell_{2} \cap B_{c_{0}},X)$ of all holomorphic functions $g: \ell_{2} \cap B_{c_{0}}\rightarrow X$ for which
 \begin{equation} \label{jungelsvectorvalued}
 \Vert g \Vert_{H_{p}(\ell_{2} \cap B_{c_{0}})} =
 \sup_{n \in \mathbb{N}} \sup_{0<r<1} \left( \int_{\mathbb{T}^{n}} \big\| g(rw_{1}, \ldots , r w_{n},0,0, \ldots) \big\|_X^{p} d(w_{1}, \ldots , w_{n})
 \right)^{\frac{1}{p}} < \infty  \,.
 \end{equation}
 In \cite{Bayart} Bayart developed an $H_p$-theory of ordinary Dirichlet series for $1\leq p\leq \infty$. This was later extended to $\lambda$-Dirichlet series in \cite{DefantSchoolmann2} through Fourier analysis on groups.
 Providing a vector-valued definition is then straightforward and gives rise to the spaces $\HH_p(\lambda,X)$ and $\HH^+_p(\lambda,X)$ which are properly defined in Section \ref{sec2vectorvalued}.

 It is proven in \cite{BaDe} that for  $1\leq  p< \infty$ there is a unique isometric equality
 \begin{equation} \label{Hpholvectorvalued}
 \mathcal{H}_{p}((\log n),X)=  H_{p}(\ell_{2} \cap B_{c_{0}},X)  \,,
 \end{equation}
 identifying Fourier and monomial coefficients if and only if $X$
 has ARNP (see also \cite[Chapter 13]{Defant}).
However, for general frequencies these characterizations are no longer available. The spaces $\HH^+_{p}(\lambda,X)$ partly compensate for this loss.

 The main goal of Section \ref{sec3vectorvalued} is to study for which frequencies $\lambda$ and Banach spaces $X$  we have
 \begin{align}\label{coincidenceIIIvectorvalued}
 \HH_{p}(\lambda,X)=\HH^+_{p}(\lambda,X),
 \end{align}
 which generalizes Question \ref{q1vectorvalued} to any $1\leq p\leq \infty$. As in the case $p=\infty$ in \eqref{arnpvectorvalued}, we for vector-valued ordinary Dirichlet series have
 \begin{align}\label{arnp2vectorvalued}
 \HH_{p}((\log n),X)=\HH^+_{p}((\log n),X),
 \end{align}
 if and only if the Banach space $X$ has ARNP (see \cite{AntonioDefant} and \cite[Theorem 24.8]{Defant}). However, we find that in the general framework the validity of \eqref{coincidenceIIIvectorvalued} depends not only on the geometry of the Banach space but also on the frequency $\lambda$. More precisely, as shown in Theorem \ref{mainresultARNPvectorvalued} and Remark \ref{cnullvectorvalued} for any frequency $\lambda$ we have
 \[X \text{ has ARNP} \Rightarrow \Hcal_p(\lambda,X)=\Hcal^+_p(\lambda,X) \text{ for all } 1\le p \le \infty \Rightarrow c_0\not\subset X.\]
 Furthermore, both extremes may characterize the coincidence of the spaces for certain frequencies.

\bigskip

\noindent {\bf Aspect III.}
Suppose that $D=\sum a_{n}e^{-\lambda_{n}s}$ (scalar coefficients) converges somewhere and that its limit function extends  to a bounded and holomorphic function $f$ on $[Re>0]$. Then a prominent problem from the beginning of the 20th century was to determine the class of $\lambda$'s for which under this assumption all $\lambda$-Dirichlet series converge uniformly on $[Re>\varepsilon]$ for every $\varepsilon>0$.

 We say that $\lambda$ satisfies 'Bohr's theorem'  if the answer to the preceding problem is affirmative. In Section \ref{Bohr's theoremvectorvalued} we are going to repeat concrete conditions on $\lambda$ under  which
Bohr's theorem holds, but also examples which show that not every frequency has this property.

 Anyway,  this notion  seems to be at the very heart of every serious study of the space  $\mathcal{D}_{\infty}(\lambda)$ -- an opinion which  is underlined
 by the equivalence of the following three statements (a result from  \cite[Theorem 5.1]{DefantSchoolmann4}):
 \begin{itemize}
\item $\lambda$ satisfies Bohr's theorem
\item $\mathcal{D}_{\infty}(\lambda)$ is complete
\item $\mathcal{D}_{\infty}(\lambda)=\mathcal{H}_{\infty}(\lambda)$
\end{itemize}
The second statement allows to apply (principles of) functional analysis to $\mathcal{D}_{\infty}(\lambda)$, whereas the third statement links its study intimately with
Fourier analysis. The proof seems highly non-trivial since it needs, among other tools,  a variant of the  Carleson theorem (on pointwise convergence of Fourier series in $H_2(\T)$) for  the Hilbert space
$\mathcal{H}_2(\lambda)$ of $\lambda$-Dirichlet series from \cite[Theorem 5.1]{DefantSchoolmann4}.

In Section \ref{sec4vectorvalued} the main aim is to find reasonable extensions of the above equivalences within the vector-valued setting. As it turns out, this phenomenon depends only on the frequency once we isolated
\eqref{motiv2vectorvalued} from \eqref{motiv1vectorvalued}. In Theorem \ref{maininftyvectorvalued} it is shown among other equivalences that
\[\mathcal{D}_{\infty}(\lambda,X)=\HH^+_{\infty}(\lambda,X),\]
if and only if $\lambda$ satisfies Bohr's theorem which answers Question \ref{q2vectorvalued}.

We in particular relate Dirichlet series in $\mathcal{D}_\infty(\lambda, X)$ (and their relatives) with the theory  of almost periodic functions on half-planes.

\bigskip \noindent {\bf Aspect IV.}
Many solutions of problems which appeared in the  modern theory of ordinary/general  Dirichlet series live from appropriate inequalities
adapted to these problems -- in particular, maximal inequalities (like in the preceding topic). One of the most prominent examples is the quantitative version of Bohr's theorem which states that for every Dirichlet series $D \in \mathcal{D}_\infty((\log n))$
and every $N\ge 2$ we have
\begin{equation}\label{quantyvectorvalued}
\big\| \sum_{n=1}^N a_n n^{-s} \big\|_\infty \leq C \log N \|D\|_\infty\,,
\end{equation}
where $C > 0$ is a universal constant. Recently, various improvements of different technical difficulties of this inequality have been proved. In \cite{Schoolmann}
a version of \eqref{quantyvectorvalued} is proved which applies to general Dirichlet series under no condition on the frequency. For certain conditions on $\lambda$ isolated by Bohr and Landau this leads to
Bohr type estimates like  \eqref{quantyvectorvalued}. More generally, Carleson-Hunt  like maximal inequalities  for (Riesz) summation of general Dirichlet series within the scale of $\mathcal{H}_p(\lambda)$-spaces have been studied
in \cite{DefantSchoolmann4} and \cite{DefantSchoolmann3}. Here it is essential to distinguish carefully between the cases $p=1$ and $1< p \leq \infty$.

In Section \ref{maximalinequalitiessecvectorvalued} we deal with our third goal, which consists in extending  a couple of fundamental maximal inequalities for  general scalar Dirichlet series to our vector-valued setting.  In some situation this just means to apply an appropriate Hahn-Banach argument, but there  are other situations where this in fact is a delicate problem since
the underlying geometry of $X$ becomes essential.

\section{Classes of general Dirichlet series}\label{sec2vectorvalued}
Recall from above the notion of  a general Dirichlet series $\sum a_n e^{-\lambda_n s}$ with coefficients in a Banach space $X$. Finite sums of the form $D=\sum_{n=1}^{N} a_{n}e^{-\lambda_{n}s}$ are called  Dirichlet polynomials, and by  $\mathcal{D}(\lambda,X)$ we denote the space of all $X$-valued $\lambda$-Dirichlet series; we write $\mathcal{D}(\lambda)$ whenever  $X =\mathbb{C}$.   It is well-known that
	if $D=\sum a_{n}e^{-\lambda_{n}s}$ converges in $s_0 \in \C$, then it also converges for all $s\in \C$
	with $Re~ s > Re~ s_0$, and its limit function $f(s) = \sum_{n=1}^{\infty} a_{n}e^{-\lambda_{n}s}$
	defines a holomorphic function on $[Re > \sigma_c(D)]$ with values in $X$, where
	  \[
  \sigma_{c}(D)=\inf\left \{ \sigma \in \R \mid D \text{ converges on } [Re>\sigma] \right\}
	 \]
	determines the so-called abscissa of convergence. The abscissas of absolute and uniform convergence
$\sigma_{a}(D)$ and $\sigma_{u}(D)$ are defined accordingly. These abscissas define the largest half-plane where the Dirichlet series converges in each sense, and the limit function $f: [Re > \sigma_c(D)] \to X$ constitutes a holomorphic function. However, for our purposes we need  summation methods of Dirichlet series $\sum a_n e^{-\lambda_n s}$ which are more general than only taking limits of  partial sums $\sum_{n=1}^N a_n e^{-\lambda_n s}$.

\subsection{Riesz means}
Fixing a frequency  $\lambda$, some $k\ge 0$, and  $D=\sum a_{n}e^{-\lambda_{n}s}\in \mathcal{D}(\lambda,X)$, the (first) $(\lambda,k)$-Riesz mean of $D$ of length $x >0$ is given by the Dirichlet polynomial
$$R_{x}^{\lambda,k}(D):=\sum_{\lambda_{n}<x} a_{n} \Big(1-\frac
{\lambda_{n}}{x}\Big)^{k}e^{-\lambda_{n}s}.$$
 We define four abscissas of $D$:
 \begin{align*}
&
\sigma^{\lambda, k}_{c}(D)=\inf\left \{ \sigma \in \R \mid (R_{x}^{\lambda,k}(D))_{x \ge 0} \text{ converges on } [Re>\sigma] \right\},
\\[1ex]&
\sigma^{\lambda, k}_{a}(D)=\inf\left \{ \sigma \in \R \mid (R_{x}^{\lambda,k}(D))_{x \ge 0} \text{ converges absolutely  on } [Re>\sigma] \right\},
\\[1ex]&
\sigma^{\lambda, k}_{u}(D)=\inf\left \{ \sigma \in \R \mid (R_{x}^{\lambda,k}(D))_{x \ge 0} \text{ converges uniformly  on } [Re>\sigma] \right\}.
\end{align*}
By definition $\sigma^{\lambda, k}_{c}(D)\le \sigma^{\lambda, k}_{u}(D)\le \sigma^{\lambda, k}_{a}(D)$, and in  general all these abscissas differ.
For historical reasons we call the following formulas 'Bohr-Cahen formulas for Riesz summation' (see \cite{HardyRiesz} and also \cite{DefantSchoolmann3}).
\begin{Prop}
\label{BohrCahenIIvectorvalued}
Let $D \in \mathcal{D}(\lambda,X)$. Then
\begin{align*}
&
\sigma^{\lambda, k}_{c}(D)\le\limsup_{x\to \infty} \frac{\log\Big(\| \sum_{\lambda_{n}<x} a_{n}\big(1-\frac
{\lambda_{n}}{x}\big)^{k}\|_{X}\Big)}{x},
\\&
\sigma^{\lambda, k}_{a}(D)\le\limsup_{x\to \infty} \frac{\log\Big( \sum_{\lambda_{n}<x} \|a_{n}\|_{X}\big(1-\frac
{\lambda_{n}}{x}\big)^{k}\Big)}{x},
\\&
\sigma^{\lambda, k}_{u}(D)\le\limsup_{x\to \infty} \frac{\log\left(\sup_{t\in \R} \| R_{x}^{\lambda,k}(D)(it)\|_{X}\right)}{x},
\end{align*}
where in each case equality holds if the left hand side is non-negative.
\end{Prop}

It is evident that  the Hahn-Banach theorem plays a fundamental role when extending results  on scalar-valued $\lambda$-Dirichlet series to valued-valued $\lambda$-Dirichlet series. For the following proposition we define for $D=\sum a_{n} e^{-\lambda_{n}s} \in \mathcal{D}(\lambda,X)$ and $x^{\ast} \in X^{\ast}$ the scalar $\lambda$-Dirichlet series
$$x^{\ast}\circ D:=\sum x^{\ast}(a_{n}) e^{-\lambda_{n}s}\in \mathcal{D}(\lambda).$$
\begin{Prop} \label{weakabscissasvectorvalued} Let $D \in \mathcal{D}(\lambda,X)$, $k \ge 0$ and let $\iota=c,u$. Then
 $$\sigma^{\lambda, k}_{\iota}(D)=\sup_{x^{\ast} \in X^{\ast}} \sigma^{\lambda, k}_{\iota}(x^{\ast}\circ D).$$
\end{Prop}
Before we prove this, let us mention that Proposition \ref{weakabscissasvectorvalued} does not hold for the abscissa of absolutely convergence. For instance choosing $X=c_{0}$ and $a_{n}=e_{n}$ (the $n$th unit vector), the ordinary Dirichlet series $D=\sum e_{n} n^{-s}$ satisfies  $\sigma_{a}(D)=1$ and $\sigma_{a}(x^{\ast}(D))=0$ for all $x^{\ast} \in X^{\ast}$ (this example is taken from \cite[Example 3.1]{Bonet}).
\begin{proof}[Proof of Proposition \ref{weakabscissasvectorvalued}]
To show that the left hand side is $\ge$ than the right hand side, is  obvious. For the other inequality assume that  $\sup_{x^{\ast} \in X^{\ast}} \sigma^{\lambda, k}_{\iota}(x^{\ast}\circ D)< \infty$ (otherwise the claim is trivial), and let $\sigma >\sup_{x^{\ast} \in X^{\ast}} \sigma^{\lambda, k}_{\iota}(x^{\ast}\circ D)$. If $\iota=c$, then  the net $\left(R_{x}^{\lambda,k}(D)(\sigma) \right)_{x\ge 0}$ is weakly bounded in $X$, and so by Mackey's theorem norm bounded. Now Proposition \ref{BohrCahenIIvectorvalued} (first formula) implies that $D$ converges on $[Re>\sigma]$ and so $\sigma_{c}(D)\le \sigma$. If $\iota=u$, then  the set
$$A(\sigma):=\left \{  \sum_{\lambda_{n}<x} a_{n} e^{-\lambda_{n} \sigma}\Big(1-\frac
{\lambda_{n}}{x}\Big)^{k} e^{-\lambda_{n} it} \mid t \in \R,~ x \in \N \right \} $$ is weakly bounded in $X$. Again by Mackey's theorem, $A(\sigma)$ is norm bounded in $X$, and the Bohr-Cahen formula for $\sigma_{u}$ (Proposition \ref{BohrCahenIIvectorvalued}, third formula) implies that $\sigma^{\lambda, k}_{u}(D)\le \sigma$.
\end{proof}

 \subsection{The spaces $\mathcal{D}_{\infty}(\lambda,X)$}
 Recall from the introduction the definition of the spaces $\mathcal{D}_{\infty}(\lambda,X)$ and  $\mathcal{D}_{\infty}(\lambda,\C)=\mathcal{D}_{\infty}(\lambda)$.
 We endow $\mathcal{D}_{\infty}(\lambda,X)$ with the supremum norm on $[Re>0]$.  In order to see that this is a norm, we need to show that every   $D\in \mathcal{D}_{\infty}(\lambda,X)$ with a limit function vanishing on  $[Re>0]$, in fact is zero, i.e. all its coefficients are zero. But a standard  application of the Hahn-Banach theorem to \cite[Corollary 3.9]{Schoolmann} shows that for every
 $D = \sum a_n e^{-   \lambda_ns}\in \mathcal{D}_{\infty}(\lambda,X)$  with limit function $f: [Re >0] \to \C$ and for all $n$,
 \[
 a_n = \lim_{T \to \infty} \frac{1}{2T}\int_{-T}^{T} f( \sigma + it) e^{(\sigma + it) \lambda_n} dt\,.
 \]
  In particular, we have
  $$\sup_{n \in \N} \|a_{n}\|_{X} \le \|D\|_{\infty}.$$
But in general $\left(\mathcal{D}_{\infty}(\lambda,X), \|\cdot\|_{\infty} \right)$ is not complete (see \cite[Theorem 5.2]{Schoolmann} for examples). Completeness
of $\mathcal{D}_{\infty}(\lambda,X)$ will be carefully studied in Section~\ref{Equivalencevectorvalued}.

\begin{Theo} \label{uniformRieszlimitvectorvalued} Let $\lambda$ be an arbitrary frequency,  and $D=\sum a_{n}e^{-\lambda_{n}s}\in \mathcal{D}_{\infty}(\lambda,X)$ with  limit function  $f$. Then for every $k>0$ and $\varepsilon>0$ we have
$$f=\lim_{x\to \infty} \sum_{\lambda_{n}<x} a_{n}\big(1-\frac{\lambda_{n}}{x}\big)^{k} e^{-\lambda_{n}s}$$
uniformly on $[Re>\varepsilon]$.
\end{Theo}

\begin{proof}
By \cite[Proposition 3.4]{Schoolmann} we know that $\sigma^{k,\lambda}_u (x^\ast \circ D) \leq 0$ for every $x^\ast \in X^\ast$, so that by Proposition~\ref{weakabscissasvectorvalued}
we have $\sigma^{k,\lambda}_u (D) \leq 0$.
\end{proof}

\begin{Theo} \label{HahnBanachinftyvectorvalued} Let $D=\sum a_{n} e^{-\lambda_{n}s} \in \mathcal{D}(\lambda, X)$. Then the following are equivalent:
\begin{enumerate}
\item[(i)] $D \in
\mathcal{D}_{\infty}(\lambda,X)$.
\item[(ii)] $x^{\ast}\circ D \in \mathcal{D}_{\infty}(\lambda)$ for all $x^{\ast}\in X^{\ast}$ and $\sup_{ \|x^{\ast}\|=1} \|x^{*}\circ D\|_{\infty}  < \infty$
\end{enumerate}
Moreover, in this case $\|D\|_{\infty}=\sup_{ \|x^{\ast}\|=1} \|x^{\ast}\circ D\|_{\infty}$.
\end{Theo}

\begin{proof}
Obviously, (i) implies (ii). Conversely,   for every $x^\ast \in X^\ast$ we by the first assumption have $\sigma^{0,\lambda}_u (x^\ast \circ D) \leq 0$, and  hence
$\sigma^{0,\lambda}_u (D) \leq 0$ by Proposition~\ref{weakabscissasvectorvalued}. This implies that $D$ converges on $[Re >0]$. By the  second assumption the limit function of $D$ is bounded
 on $[Re >0]$.
\end{proof}

\subsection{The spaces $\mathcal{H}_{\infty}^{\lambda}([Re>0],X)$}
In the following we define almost periodicity  for  functions on the real line or half planes with values in a Banach spaces -- all our definitions are straightforward extensions of the well-known definitions
for complex-valued functions (see e.g. \cite{Besicovitch}).

 Given a Banach space $X$, a  continuous function $g:\R \to X$ is said to be uniformly almost periodic (compare  \cite[pp.1-2]{Besicovitch} for $X = \C$)
 if for  every $\varepsilon>0$ there is a number $l>0$ such that for all intervals $I\subset \R$ with $|I|=l$ there is  $\tau \in I$  such that
$$\sup_{ x \in \R} \|g(x +\tau)-g(x)\|_X<\varepsilon.$$
Let now $F: [Re >0] \to X$ be a bounded and holomorphic function such that for all $\sigma>0$ the restriction $t \mapsto F(\sigma+it)$ to the vertical line $[Re=\sigma]$ is uniformly almost periodic. Fixing $x \in \R$ and  $\sigma>0$, the limit
$$a_{x}(F):=\lim_{T\to \infty} \frac{1}{2T} \int_{-T}^{T} F(\sigma+it) e^{(\sigma+it)x} dt$$
exists and  is independent of the choice of $\sigma$; we call it the  $x$th Bohr coefficient of $F$.
All this follows from the scalar case discussed in  \cite[p. 147]{Besicovitch}) and a straightforward application
of the Hahn-Banach theorem. For all  $\sigma >0 $ and $x \in \R$ we have that
\begin{equation*} \label{slimvectorvalued}
|a_{x}(F)|\le e^{\sigma x} \|F\|_{\infty}\,,
\end{equation*}
and hence  $|a_{x}(F)|\le \|F\|_{\infty}$ for $x \in \R$  and $a_{x}(F) = 0$ for  $x<0$.
   Moreover, at most countable many Bohr coefficients are non zero, and  $F$ vanishes, whenever its Bohr coefficients vanish (compare with  \cite[p. 148 and p. 18]{Besicovitch}).

 Note that the typical examples of such functions are finite polynomials $F(z):=\sum_{n=1}^{N} a_{n} e^{-\lambda_{n}z}$ with coefficients $0 \neq a_n \in X$ and frequencies $\lambda_n \ge 0$, and then the $a_{k}$ are precisely the (non-zero) Bohr coefficients of $F$.

 The following definition will be important. Given a frequency $\lambda$ and a Banach space $X$,
 $$\mathcal{H}_{\infty}^{\lambda}([Re>0],X)$$
 consists  of all bounded and holomorphic functions $F\colon [Re>0] \to X$, which are  almost periodic on all  abscissas $[Re=\sigma], \sigma>0,$ and for which the Bohr coefficients $a_{x}(F)$ vanish whenever $x \notin \{\lambda_{n} \mid n\in \mathbb{N}\}$.

Together with $\|F\|_{\infty}:=\sup_{z\in [Re>0]} \|F(z)\|_{X}$ the linear space $\mathcal{H}_{\infty}^{\lambda}([Re>0],X)$ becomes a Banach space, and we call this class of spaces 'Besicovitch spaces'.

 \subsection{The spaces $\mathcal{H}_p(\lambda,X)$}
\label{HplambdaX}
 From \cite{DefantSchoolmann2} we recall the definition and some basic facts of so-called Dirichlet groups.
 Let $G$ be a compact abelian group and $\beta\colon (\R,+) \to G$ a homomorphism of groups.  Then the pair $(G,\beta)$ is called Dirichlet group, if $\beta$ is continuous and has dense range. In this case  the dual map $\widehat{\beta}\colon \widehat{G} \hookrightarrow \R$ is injective, where we identify $\mathbb{R}=\widehat{(\mathbb{R},+)}$ (note that we do not assume $\beta$ to be injective).

   \smallskip
  Consequently, the  characters $e^{-ix\pmb{\cdot}} \colon \R \to \T$, $x\in \widehat{\beta}(\widehat{G})$, are precisely those which define  a unique $h_{x} \in \widehat{G}$  such that $h_{x} \circ \beta=e^{-ix\pmb{\cdot}}$. In particular, we have that
\begin{equation*}
\widehat{G}=\{h_{x} \mid x \in \widehat{\beta}(\widehat{G}) \}.
\end{equation*}
From \cite[Section 3.1]{DefantSchoolmann2} we know that every $L_{1}(\R)$-function may be interpreted as a bounded regular Borel measure on $G$. In particular, for every $u>0$ the Poisson kernel
$$P_{u}(t):=\frac{1}{\pi}\frac{u}{u^{2}+t^{2}}\,,\,\,\, t \in \R,$$
defines a measure $p_{u}$ on $G$, which we call the Poisson measure on $G$. We have $\|p_{u}\|\le\|P_{u}\|_{L_{1}(\R)}=1$ and
\begin{equation*}\label{Fourier1Helsonvectorvalued}
\text{$\widehat{p_{u}}(h_{x})=\widehat{P_{u}}(x)=e^{-u|x|}$ for all $u >0$ and
$x\in \widehat{\beta}(\widehat{G})$.}
\end{equation*}
Now, given a frequency $\lambda$, we call a Dirichlet group $(G,\beta)$ a $\lambda$-Dirichlet group whenever $\lambda \subset \widehat{\beta}(\widehat{G})$, or equivalently whenever for every
$e^{-i\lambda_{n} \pmb{\cdot}} \in \widehat{(\mathbb{R},+)}$ there is (a unique) $h_{\lambda_{n}}\in  \widehat{G}$ with $h_{\lambda_{n}}\circ \beta=e^{-i\lambda_{n} \pmb{\cdot}}$.

 Note that for every $\lambda$ there exists a $\lambda$-Dirichlet groups $(G,\beta)$ (which is not unique).
To see a very first example, take the Bohr compactification $\overline{\R}$ together with the mapping
$$\beta_{\overline{\R}} \colon \R \to \overline{\R}, ~~ t \mapsto \left[ x \mapsto e^{-itx} \right].$$
Then $\beta_{\overline{\R}}$ is continuous and has dense range (see e.g. \cite[Theorem 1.5.4, p. 24]{QQ} or \cite[Example 3.6]{DefantSchoolmann2}), and so the pair $(\overline{\R},\beta_{\overline{\R}})$ forms a $\lambda$-Dirichlet group for all $\lambda$'s. We refer to \cite{DefantSchoolmann2}  for more 'universal' examples of Dirichlet groups. Looking at the  frequency  $\lambda=(n)=(0,1,2,\ldots)$, the group $G=\T$ together with \[\beta_\T: \R \to \T, \,\,\beta_{\T}(t)=e^{-it},\]
forms  a $\lambda$-Dirichlet group, and  the so-called
Kronecker flow
\begin{equation*}
\label{oscarHelsonvectorvalued}
\beta_{\T^{\infty}}\colon \R \to \T^{\infty}, ~~ t \mapsto \mathfrak{p}^{-it}=(2^{-it},3^{-it}, 5^{-it}, \ldots),
\end{equation*}
turns the infinite dimensional torus $\T^{\infty}$ into a  $\lambda$-Dirichlet group
for $\lambda = (\log n)$.
We note that, identifying  $\widehat{\T} = \Z$ and $\widehat{\T^\infty} = \Z^{(\N)}$ (all finite sequences of integers), in the first case $h_n(z) = z^n$ for $z \in \T, n \in \Z$,
and in the second case $h_{\sum \alpha_j \log p_j}(z) = z^\alpha$ for  $z \in \T^\infty, \alpha \in \Z^{(\N)}$.

We finish with another   crucial tool  given by the following fact from \cite[Lemma 3.10]{DefantSchoolmann3}: For any  $\lambda$-Dirichlet group $(G,\beta)$ and $k>0$ there is a constant $C=C(k)>0$ such that for all $x>0 $ there is a measure $\mu_x \in M(G)$ which satisfies  $\|\mu_x\|\le C$ and for all $n$
\begin{equation} \label{measureRieszmeanvectorvalued}
\widehat{\mu_x}(h_{\lambda_{n}})=\begin{cases} \big(1-\frac{\lambda_{n}}{x}\big)^{k},&  ~~\lambda_{n}< x,\\
0,&  ~~\lambda_{n}\ge x. \end{cases}
\end{equation}

 In the following we are going to extend several results from \cite{AntonioDefant} from the ordinary to the general case. But many arguments in the ordinary case rely on the good properties of  the Poisson kernel
 \begin{equation} \label{humerusvectorvalued}
 p_N: \mathbb{D}^N \times \mathbb{T}^N \to \C\,, \,\, p_N(z,w) = \sum_{\alpha \in \Z^N} w^{-\alpha} |z|^{|\alpha|} \Big(\frac{z}{|z|}\Big)^\alpha\,.
 \end{equation}
 In  our much more general setting of general Dirichlet series and Dirichlet groups, this fundamental tool is not available -- but in many cases the measures from \eqref{measureRieszmeanvectorvalued} will be an appropriate substitute.

 Let us turn to Hardy spaces of $\lambda$-Dirichlet series.
Fix some $\lambda$-Dirichlet group $(G,\beta)$, a Banach space $X$, and $1\le p \le \infty$. By
 $$H_{p}^{\lambda}(G,X)$$
 we denote the  Hardy space of all functions
$f\in L_{p}(G,X)$ (the Banach space  of all  $X$-valued $p$-Bochner integrable functions on $G$) having a Fourier transform
 \[
 \widehat{f}(\gamma) = \int_G f(\omega) \overline{\gamma}(\omega) dm(\omega), \, \gamma \in \widehat{G},
 \]
 supported on $\{h_{\lambda_n} \colon n \in \mathbb{N}\} \subset \widehat{G}$. Being a closed subspace of $L_p(G,X)$, this clearly defines a Banach space.

With the  spaces $H_{p}^{\lambda}(G,X)$ at hand we in a natural way  define $\mathcal{H}_p$'s of  $\lambda$-Dirichlet series. Let
$$\Hcal_{p}(\lambda,X)$$
be the class of all $\lambda$-Dirichlet series $D=\sum a_n e^{-\lambda_n s}$ for which there is some
$f \in H_p^\lambda(G,X)$ such that   $a_n = \widehat{f}(h_{\lambda_{n}})$ for all $n$. In this case the function $f$ is unique, and together with
the norm
$$\|D\|_{p}:=\|f\|_{p}$$
the linear space  $\Hcal_{p}(\lambda,X)$ obviously  forms a Banach space. So (by definition) the so-called Bohr map
\begin{equation} \label{BohrmapHelsonvectorvalued}
\Bcal\colon H_{p}^{\lambda}(G,X)\to \mathcal{H}_{p}(\lambda,X),~~ f \mapsto \sum \widehat{f}(h_{\lambda_{n}}) e^{-\lambda_{n}s},
\end{equation}
defines an onto isometry. A fundamental fact (extend the proof of  \cite[Theorem 3.24.]{DefantSchoolmann2}
word by word to the vector-valued case) is that the definition of $\mathcal{H}_{p}(\lambda,X)$ is independent of the chosen $\lambda$-Dirichlet group $(G,\beta)$.

Our two basic examples of frequencies, $\lambda = (n)$ and $\lambda = (\log n)$, lead to well-known examples:
\begin{equation} \label{hardyTHelsonvectorvalued}
   H_{p}(\T,X)=H_{p}^{(n)}(\T,X) \,\,\, \,\text{and} \,\,\,\,   H_p(\T^\infty,X) = H_p^{(\log n)}(\T^\infty,X) \,.
\end{equation}
In particular,
$f \in H_p^{(n)}(\T,X)$ if and only if $f \in L_p(\T,X)$ and $\widehat{f}(n) = 0$ for any $n \in \Z$ with  $n < 0$, and
$f \in H_p^{(\log n)}(\T^\infty,X)$ if and only if $f \in L_p(\T^\infty,X)$ and $\widehat{f}(\alpha) = 0$ for any finite sequence
$\alpha = (\alpha_k)$ of integers with  $\alpha_k < 0$ for some $k$
(where as usual  $\widehat{f}(\alpha) := \widehat{f}(h_{\log \mathfrak{p}^\alpha}))$.
Consequently, if we turn to Dirichlet series, then the Banach spaces
$$\mathcal{H}_p(X)= \mathcal{H}_p((\log n),X)$$
are precisely  Bayart's Hardy spaces of ordinary $X$-valued Dirichlet series from \cite{Bayart}
(see also \cite{Defant} and \cite{QQ}).

We use ideas from the proof of \cite[Theorem 3.26]{DefantSchoolmann2} to give the following   internal description of $\mathcal{H}_{p}(\lambda,X)$, $1\le p <\infty$: The limit
\begin{equation}\label{internalvectorvalued}
\|D\|_{p}:=\lim_{T\to \infty} \Big(\frac{1}{2T} \int_{-T}^{T} \big\| \sum_{n=1}^{N}a_{n}e^{-\lambda_{n}it}\big\|^{p} dt \Big)^{\frac{1}{p}}
\end{equation}
exists, since we integrate an almost periodic function on $\R$, and it defines a norm on the space $Pol(\lambda,X)$ of all $X$-valued $\lambda$-Dirichlet polynomials. Then $\mathcal{H}_{p}(\lambda, X)$ is the completion of $\big(Pol(\lambda,X),\|\cdot\|_{p}\big)$; here the density of $Pol(\lambda,X)$ in $\mathcal{H}_{p}(\lambda, X)$ follows by an analysis of the arguments in the scalar case, that are given in
\cite[Proposition 3.14]{DefantSchoolmann2}. We for the sake of completeness sketch the proof: Fix some $(G,\beta)$ be a Dirichlet group and  $A\subset \widehat{G}$. Denote by $C_{A}(G,X)$ the Banach space of all continuous functions $f\colon G \to X$ with Fourier transform supported on $A$ and by $Pol_{A}(G,X)$ the set of all $X$-valued polynomials of the form $\sum_{\gamma \in A} x_{\gamma} \gamma$. Once we prove that $Pol_{A}(G,X)$ is dense in  $C_{A}(G,X)$ with respect to the sup norm, the full claim follows the same lines as in \cite[Proposition 3.14]{DefantSchoolmann2} with $A=\{ h_{\lambda_{n}} \mid n\in \N\}$. This fact we prove  by contradiction, so assume that there exists some $g\in C_{A}(G,X) \setminus \overline{Pol_{A}(G,X)}$. Then by the theorem of Hahn-Banach  there is some  $\varphi\in C(G,X)^{\ast}$ with $\varphi(g)\ne 0$, that vanishes on $Pol_{A}(G,X)$. Moreover, a direct calculation shows that the continuous function
$\varphi*g(x):=\varphi(g(x+\cdot))\colon G\to X$ satisfies $\widehat{\varphi*g}(\gamma)=\varphi(\widehat{g}(\gamma)\gamma)$ for every $\gamma \in \widehat{G}$. Then we easily deduce that $\widehat{\varphi*g}=0$, and  consequently $\varphi*g=0$ (here again the Hahn-Banach theorem is needed). In particular we have $0=\varphi*g(0)=\varphi(g)\ne 0$, a contradiction.

\subsection{The spaces $\mathcal{H}^{+}_p(\lambda,X)$}
We already remarked in the introduction that when we do the jump from ordinary Dirichlet series to general ones, then we
pay  the price to lose holomorphic functions in infinitely dimensional polydiscs. Here we define a class of spaces that can compensate for this loss.

Recall that there is an isometric coefficient preserving equality
\[
 H_p(\mathbb{T}, X) =H_p(\mathbb{D},X)
\]
if and only if $X$ has ARNP (see e.g. \cite[Theorem 23.6]{Defant}). The preceding result then says that $X$ has ARNP
if and only if:  $f \in H_p(\mathbb{T},X)$ if and only if $f  \ast p(r, \cdot ) \in H_p(\mathbb{T},X)$ for all $0 < r < 1$, and in this case
\[
\|f\|_p = \sup_{0<r <1} \|f  \ast p(r, \cdot ) \|_p\,.
\]
(recall from \eqref{humerusvectorvalued} the definition of the Poison kernel $p=p_1$).
If we translate this in terms of scalar $(n)$-Dirichlet series, then it reads:   $D \in \mathcal{H}_p((n))$  if and only if all translates
$D_\sigma = \sum a_n e^{-n \sigma } e^{-n s}\in \mathcal{H}_p((n)),\, 0 < \sigma < \infty$, and in this case $$\|D\|_p = \sup_{0<\sigma<\infty}\|D_\sigma\|_p\,.$$
Let now $\lambda$ be an arbitrary  frequency, $X$ a Banach space, and $1\le p \le  \infty$.
Inspired by the 'ordinary definition' from \cite{AntonioDefant} (see also \cite[Chapter 24.3]{Defant}), we  define   the Banach space
$$\Hcal_{p}^{+}(\lambda,X)$$
as the space of all (formal) $X$-valued $\lambda$-Dirichlet series $D$
such that the translate $D_{\sigma}=\sum a_{n}e^{-\sigma \lambda_{n}}e^{-\lambda_{n}s} \in \Hcal_{p}(\lambda,X)$  for all $\sigma>0$ and $\|D\|^{+}_{p}:=\sup_{\sigma>0} \|D_{\sigma}\|_{p}<\infty$. It is worth mentioning that $\|D_{\sigma}\|_{p}\nearrow\|D\|^{+}_{p}$ as $\sigma\to0$, whenever $1\le p<\infty$. The proof of this is analogous to \cite[Proposition 2.3]{AntonioDefant} and is sketched in Lemma \ref{blavectorvalued}.
We will see that this is the space which  must take over the role holomorphic functions in infinite dimensions play in the ordinary world.

\begin{Prop}\label{Hpplusscalarvectorvalued} For every $1\le p \le \infty$ and frequency $\lambda$  the isometric equality
$\mathcal{H}_{p}(\lambda)=\mathcal{H}_{p}^{+}(\lambda)$ holds.
\end{Prop}

\begin{proof}
Although parts of this result are proved in \cite{DefantSchoolmann2}, we prefer to sketch the argument. For $1 \leq p < \infty$ the  embedding $\mathcal{H}_{p}(\lambda) \subset \mathcal{H}_{p}^{+}(\lambda)$ is isometric -- this is proved in \cite[Theorem 4.7]{DefantSchoolmann2}.
The proof extends to $p=\infty$ since for each $D \in \mathcal{H}_{\infty}(\lambda)$, we have that
$
\|D\|_\infty^+ = \lim_{p\to \infty} \|D\|_p^+
$
(alternatively, see the proof of the more general result Corollary~\ref{inclusionplusvectorvalued}).
Conversely, for $1 < p \leq \infty$ the proof of $\mathcal{H}_{p}^{+}(\lambda) \subset \mathcal{H}_{p}(\lambda)$
follows from a standard weak-compactness  in $L_{q}(G)^\ast$, where $(G,\beta)$ is an appropriate
$\lambda$-Dirichlet group and $\frac{1}{p}+\frac{1}{p} =1$ (compare the proof for the ordinary case given in
\cite[Proof of Theorem 11.21]{Defant}). The case $p=1$ is proved in
\cite[Theorem 4.7]{DefantSchoolmann2}; the proof uses \cite[Lemma 4.9]{DefantSchoolmann2} which has a further assumption on $\lambda$, but an easy analysis shows that this assumption is in fact not needed.
Alternatively, see the proof of the more general result Theorem~\ref{mainresultARNPvectorvalued}.
\end{proof}

We finish with  a characterization of Dirichlet series
$\mathcal{H}_p^+(\lambda, X)$ in terms of Riesz means -- an important tool in many of  the forthcoming proofs.

\begin{Prop} \label{tool1vectorvalued} Let  $k>0$,  $1\le p \le \infty$, and $D=\sum a_{n}e^{-\lambda_{n}s} \in \mathcal{D}(\lambda, X)$. Then $D \in \mathcal{H}_p^+(\lambda, X)$
if and only if
\begin{equation*}
M_{k,p}(D) := \sup_{x>0} \big\| \sum_{\lambda_{n}<x} a_{n}\big(1-\frac{\lambda_{n}}{x}\big)^{k} e^{-\lambda_{n}s}\big\|_{p} < \infty\,,
\end{equation*}
and in this case for all $\sigma >0$
$$D_\sigma = \lim_{x\to \infty} \sum_{\lambda_{n}<x} a_n \big(1-\frac{\lambda_{n}}{x}\big)^{k} e^{-\sigma\lambda_{n}} e^{-\lambda_{n} s}
$$
with convergence in $\mathcal{H}_{p}(\lambda,X)$.
Moreover, for every $k>0$ there is a constant $C_{k}>0$ such that for every $1\le p \le \infty$ and $D \in \mathcal{H}_p^+(\lambda, X)$
\begin{equation*} \label{formulaRieszvectorvalud}
\|D\|_{p}^{+}\le  M_{k,p}(D) \le C_{k} \|D\|_{p}^{+};
\end{equation*}
here for $0<k\le 1$  the choice $C_{k}=\frac{C}{k}$ with some absolute constant $C$ is possible.
\end{Prop}

\begin{proof}
Assume first that $M_{k,p}(D) < \infty$, and define  $E=\sum (a_{n} e^{-\lambda_{n}s})e^{-\lambda_{n}z} \in \mathcal{D}(\lambda, \mathcal{H}_{p}(\lambda,X))$. Then the Bohr-Cahen formula for $\sigma_{c}^{\lambda, k}$ from Proposition \ref{BohrCahenIIvectorvalued} implies  that $\sigma_{c}^{\lambda, k}(E)\le 0$, that is for every $\sigma>0$ the limit
\begin{equation*} \label{jorinvectorvalued}
\lim_{x\to \infty}\sum_{\lambda_{n}<x} (a_{n}e^{-\lambda_{n}s})\big(1-\frac{\lambda_{n}}{x}\big)^{k} e^{-\sigma \lambda_{n}}
\end{equation*}
exists in $\mathcal{H}_{p}(\lambda,X)$, and the limit is given by $D_{\sigma}$. In particular, by convolution with the Poisson measure we obtain for every $\sigma>0$
\begin{align*}
\|D_{\sigma}\|_{p}&=\lim_{x\to \infty} \|\sum_{\lambda_{n}<x} (a_{n}e^{-\lambda_{n}s})\big(1-\frac{\lambda_{n}}{x}\big)^{k} e^{-\sigma \lambda_{n}}e^{-\lambda_{n}s}\|_{p} \\ & =\lim_{x\to \infty} \|\big(\sum_{\lambda_{n}<x} (a_{n}e^{-\lambda_{n}s})\big(1-\frac{\lambda_{n}}{x}\big)^{k} h_{\lambda_{n}}\big)*p_{\sigma}\|_{p} \\ &\le\lim_{x\to \infty} \|\sum_{\lambda_{n}<x} (a_{n}e^{-\lambda_{n}s})\big(1-\frac{\lambda_{n}}{x}\big)^{k}e^{-\lambda_{n}s}\|_{p} \le M_{k,p}(D),
\end{align*}
and so $\|D\|_{p}^{+}\le M_{k,p}(D).$
Suppose conversely that  $D\in  \mathcal{H}_{p}^{+}(\lambda,X)$. Then for every $\sigma>0$ we have $D_{\sigma} \in \mathcal{H}_{p}(\lambda,X)$ with $\|D_{\sigma}\|_p \leq \|D\|^+_p$. Hence, if
$f_\sigma \in H_p^\lambda(G, X)$ is associated with $D_\sigma$ (where $(G, \beta)$ is some appropriate $\lambda$-Dirichlet group), then  by
 \eqref{measureRieszmeanvectorvalued} for every $x$
\begin{align*}
\big\| \sum_{\lambda_{n}<x} a_{n}e^{-\sigma\lambda_{n}}
&
\big(1-\frac{\lambda_{n}}{x}\big)^{k} e^{-\lambda_{n}s}\big\|_{p}
=
\big\| \sum_{\lambda_{n}<x} a_{n}e^{-\sigma\lambda_{n}}
\big(1-\frac{\lambda_{n}}{x}\big)^{k} h_{\lambda_n}\big\|_{p}
\\&
=\|f_{\sigma}*\mu_{x}\|_{p}
\le \|f_{\sigma}\|_{p}\|\mu_{x}\| =  \|D_{\sigma}\|_{p}\|\mu_{x}\|\le  C_{k}\|D\|_{p}^{+},
\end{align*}
where $C_{k}$ only depends on $k$ and for $0<k\le 1$ the choice $C_{k}=\frac{C}{k}$ with an absolute constant $C$ is possible (see \cite[Lemma 3.10, Proposition 3.2 and Theorem 2.1]{DefantSchoolmann3}). If now in the preceding estimate $\sigma \to 0$, then we
get as desired that $ M_{k,p}(D)\leq C_{k} \|D\|_{p}^{+}$. This completes the proof.
\end{proof}

Let us revisit the case $\lambda=(n)$ with $(n)$-Dirichlet group $(\T,\beta_{\T})$. Then, given $f\in H_{\infty}(\T)$ and $\sigma>0$, we have that
\begin{equation*}
f*p_{\sigma}(z)=\sum_{n=0}^{\infty} \widehat{f}(n) e^{-\sigma n} z^{n}
\end{equation*}
uniformly on $\T$, and in particular we know that $f*p_{\sigma}$ is continuous on $\T$. The following corollary  extends this observation to the case of  general $\lambda$'s.

\begin{Coro}
\label{heutevectorvalued}
Let $(G, \beta)$ be a $\lambda$-Dirichlet group and  $f\in H_\infty^\lambda(G, X)$.
Then for every $k, \sigma>0$ we have uniformly on $G$
$$f*p_{\sigma}= \lim_{x\to \infty} \sum_{\lambda_{n}<x} \widehat{f}(h_{\lambda_{n}}) \big(1-\frac{\lambda_{n}}{x}\big)^{k} e^{-\sigma\lambda_{n}} h_{\lambda_{n}}.$$
In particular, the function $f*p_{\sigma}$ is continuous for every $\sigma>0$.
\end{Coro}

\section{General Dirichlet series vs.  operators}\label{sec3vectorvalued}

As mentioned in the introduction, the main goal of this section is to study the coincidence of the spaces $\HH_p(\lambda,X)$ and $\HH^+_p(\lambda,X)$. As an application we obtain a generalization of brothers Riesz theorem and conditions for $\HH_p(\lambda,X^*)$ to be a dual space. A useful tool with interest in its own will be to identify $\HH^+_p(\lambda,X)$ with a suitable space of cone summing operators.

\subsection{Cone summing operators}
Let $1 \leq p,q \leq \infty$ and  with $\frac{1}{p}+\frac{1}{q}=1$.
To simplify, we  from now on fix  a $\lambda$-Dirichlet group $(G, \beta)$ with Haar measure $m$. Then  for  $1 \le q < \infty$ we write $E_{q}(G)= L^{q}(G)$ and  $E_q(G) =C(G)$ for $q= \infty$.
Moreover, let $X$ be some Banach space.

Recall that a (bounded linear) operator $T\colon E_q(G) \to X$ is  cone summing (in short $T\in \Pi_{cone}(E_q(G),X)$)
if there is a constant $C>0$ such that for every choice of finitely many we have $f_{1}, \ldots , f_{N} \in E_q(G)$
$$\sum_{n=1}^{N} \|T(f_{k})\|_{X} \le C \Big\| \sum_{n=1}^{N} |f_{k}| \Big\|_{E_q(G)},$$
and in this case $\pi_{cone}(T):=\inf C$ is the so-called cone summing norm. To see more examples recall the  following classical results
\begin{equation} \label{cone1vectorvalued}
\Pi_{cone}(L_1(G),X)=\mathcal{L}(L_1(G),X),
\end{equation}
and
\begin{equation} \label{cone2vectorvalued}
\Pi_{cone}(C(G),X)=\Pi_{1}(C(G),X),
\end{equation}
where $\Pi_1$ indicates all  summing operators from $C(G)$ into $X$.
For  all needed information on cone summing operators in the context of Dirichlet series,  see \cite[24.4]{Defant} and \cite{AntonioDefant}.

\begin{Rema}\label{sylviavectorvalued}
From \cite[Proposition 24.16]{Defant} we deduce that the canonical inclusion
$$ L_{p}(G,X) \hookrightarrow \Pi_{cone}(E_{q}(G),X), ~~ f \mapsto \Big[ g \mapsto \int f g~ dm \Big]$$
defines an into isometry.
\end{Rema}
 Given $T \in \Pi_{cone}(E_{q}(G),X)$, we call $$\widehat{T} \colon \widehat{G} \to X, ~~\gamma\mapsto T(\overline{\gamma}),$$ the Fourier transform of $T$.
With this definition we see by density of the polynomials in $E_q(G)$  that $T$ is uniquely determined by its Fourier coefficients, that is $T=0$, whenever $\widehat{T}=0$.

\begin{Rema}
\label{convolutionvectorvalued}
Let $T \in\Pi_{cone}(E_{q}(G),X)$ and $\mu \in  M(G)$. Then  $$(T*\mu)(g):= T(g \ast \mu)\colon E_{q}(G) \to X $$ belongs to $\Pi_{cone}(E_{q}(G),X)$ with $\pi_{cone}(T*\mu)\le \pi_{cone}(T)\|\mu\|$ and $\widehat{T*\mu}= \widehat{T} \widehat{\mu}$.
\end{Rema}
\begin{proof}
Using the  definition of cone summing operators the proof of the first two statements are straightforward.  For the formula on Fourier coefficients check
\begin{align*}
\widehat{T*\mu}(\gamma)
&
=(T*\mu)(\overline{\gamma})
=T\Big(\int_{G} \overline{\gamma(\cdot+x)} d\mu(x)\Big)
\\&
=T\Big(\int_{G} \overline{\gamma}(\cdot)  \overline{\gamma}(x)  d\mu(x)\Big)
=T\Big( \overline{\gamma}(\cdot) \int_{G} \overline{\gamma}(x)  d\mu(x)\Big) =\widehat{\mu}(\gamma) \widehat{T}(\gamma) \,,
\end{align*}
as desired.
\end{proof}

For  $1\leq p\le \infty$  we denote by $$\Pi^{\lambda}_{cone}(E_{q}(G), X)$$ the Banach space of all $T  \in \Pi_{cone}(E_{q}(G),X)$ such that
 $\widehat{T}(\gamma)\ne 0$ implies  $\gamma =h_{\lambda_{n}}$ for some $n$.
Now operators $T\in \Pi^{\lambda}_{cone}(E_{q}(G), X)$ formally define $\lambda$-Dirichlet series via the mapping
\begin{equation*}
T \mapsto   \sum \widehat{T}(h_{\lambda_{n}}) e^{-\lambda_{n}s}\,,
\end{equation*}
and then our main result in this section shows that this  identifies
 $\Pi^{\lambda}_{cone}(E_{q}(G), X)$ with $\mathcal{H}_{p}^{+}(\lambda,X)$. In the ordinary case $\lambda = (\log n)$ this was proved in \cite[Theorem~4.2]{AntonioDefant} (see also \cite[Theorem~24.13]{Defant}).
\begin{Theo} \label{conesummingcoincidencevectorvalued} Let $\lambda$ be arbitrary, $(G, \beta)$ a $\lambda$-Dirichlet group and $X$ a Banach space. Then for every $1\le p,q \le \infty$ with $\frac{1}{p}+\frac{1}{q}=1$ the map
$$ \mathcal{T} \colon  \Pi^{\lambda}_{cone}(E_{q}(G),X) \to \Hcal^{+}_{p}(\lambda,X), ~~ T \mapsto \sum \widehat{T}(h_{\lambda_{n}}) e^{-\lambda_{n}s}$$ is an onto isometry.
\end{Theo}

Our proof is inspired by the proof of the ordinary case -- nevertheless it is substantially different since, among others, no Poisson kernel is available in the case of general Dirichlet series.
\begin{proof}
Let $T \in \Pi^{\lambda}_{cone}(E_{q}(G),X)$, and $\mu_{x}$ the measure from (\ref{measureRieszmeanvectorvalued}) with $k=1$ and $x >0$. We define $D=\sum a_{n}e^{-\lambda_{n}s}$, where $a_{n}:=\widehat{T}(h_{\lambda_{n}})$ for all $n$. Then by Remark~\ref{sylviavectorvalued} and Remark~\ref{convolutionvectorvalued} we have
$$\Big\| \sum_{\lambda_{n}<x} a_{n}\big(1-\frac{\lambda_{n}}{x}\big)e^{-\lambda_{n}s} \Big\|_{p}=\pi_{cone}(T*\mu_{x})\le C_{1}\pi_{cone}(T)\,,$$
and by Proposition~\ref{tool1vectorvalued} we conclude that $D=\sum a_{n} e^{-\lambda_{n}s} \in \mathcal{H}_{p}^{+}(\lambda,X)$ with $\|D\|^{+}_{p}\le \pi_{cone}(T)$. Suppose conversely  that
$D= \sum a_{n} e^{-\lambda_{n}s}\in \mathcal{H}_{p}^{+}(\lambda,X)$ and let $f_{\sigma} \in H_{p}^{\lambda}(G,X)$ correspond to
$D_{\sigma} = \sum a_{n} e^{-\lambda_{n}\sigma} e^{-\lambda_{n}s}, \sigma>0$. We define an operator
$T\colon Pol(G) \to X$ by
$$\widehat{T}(h_{x})= \begin{cases} a_{n}, & \text{if } x= \lambda_{n} \text{ for some } n, \\ 0, & else.\end{cases}$$ Then we claim that there is $g \in E_{q}(G)^{\ast}$ with $\|g\|_{E_{q}(G)^{\ast}}\le\|D\|_{p}^{+}$ such that $\|T(P)\|_{X}\le g( |P| )$ for all $P \in Pol(G)$. This completes the proof, since by density of the polynomials and the Pietsch type domination theorem for cone summing operators (see e.g. \cite[Proposition 24.12]{Defant}) we obtain that $T$ is cone summing with $\pi_{cone}(T)\le \|g\|_{E_{q}(G)^{\ast}}$, and so $\mathcal{T}(T)=D$ and $\pi_{cone}(T)\le \|D\|_{p}^{+}$. Indeed, the desired $g \in E_{q}(G)^{\ast}$ exists: Since for all $x \in \R$
$$T(\overline{h_{x}})e^{-|x|\sigma}=\hat{f_{\sigma}}(x)=\int_{G} f_{\sigma}(\omega)~ \overline{h_{x}(\omega)}~ d\omega, $$
we obtain
$$T(P)= \int_{G} f_{\sigma}(\omega)~ \sum c_{x}e^{|x|\sigma} h_{x}(\omega)~ d\omega$$
for all polynomials $P=\sum c_{x} h_{x} \in Pol(G)$. Hence
$$\|T(P)\|_{X}\le  \int_{G} \|f_{\sigma}(\omega)\|_{X} \big| \sum c_{x}e^{|x|\sigma} h_{x}(\omega) \big| d \omega.$$
Since the net $(\|f_{\sigma}(\cdot)\|_{X})_{\sigma} \subset L_{p}(G) \subset E_{q}(G)^{\ast}$ is a weak-star bounded, there is a sequence $\sigma_{j}\to 0$ and $g \in E_{q}(G)^{\ast}$ such that $\|f_{\sigma_{j}}(\cdot)\|_{X} \to g$ with respect to the weak-star topology and $\|g \|\le \|D\|_{p}^{+}$. This implies that
$$\|T(P)\|_X\le \lim_{j \to \infty} \int_{G} \|f_{\sigma_{j}}(\omega)\|_{X} \big| \sum c_{x}e^{|x|\sigma_{j}} h_{x}(\omega) \big| d\omega = g\left( \big| \sum c_{x} h_{x} \big| \right),$$
and so as desired $\|T(P)\|_{X}\le g(|P| )$ for all polynomials $P\in Pol(G)$.
\end{proof}

From the into isometry in Remark~\ref{sylviavectorvalued} we immediately see that the inclusion
 $$\mathcal{H}_{p}(\lambda,X)=H_{p}^{\lambda}(G,X)\subset \Pi^{\lambda}_{cone}(E_{q}(G), X)$$
 is isometric as well. Hence Theorem~\ref{conesummingcoincidencevectorvalued} immediately implies the following consequence.
\begin{Coro} \label{inclusionplusvectorvalued} Let $\lambda$ be a frequency and $X$ a Banach space. Then for every $1\le p \le \infty$ isometrically $\mathcal{H}_{p}(\lambda,X)\subset \mathcal{H}^{+}_{p}(\lambda,X).$
\end{Coro}

\subsection{Nth abschnitte}
Hilbert's criterion from \cite[Theorem 15.26, p. 372]{Defant} states that to a family $(c_{\alpha})_{\alpha \in \N_{0}^{(\N)}}$ in a Banach space $X$ there is a function $f\in H_{\infty}(B_{c_{0}},X)$ such that $c_{\alpha}= \frac{\partial f(0)}{\alpha!}$ for every $\alpha\in \N_{0}^{(\N)}$ if and only if
$$\sup_{N\in \N} \sup_{z\in \D^{N}} \big\| \sum_{\alpha \in \N_{0}^{N}} c_{\alpha}z^{\alpha}\big\|_{X}<\infty.$$
In view of (\ref{hol_pvectorvalued}) this results translates to ordinary Dirichlet series. In fact, an extension of this result to the framework of vector-valued $\lambda$-Dirichlet series is possible, which requires no assumption on the frequency $\lambda$.

Therefore, note that every frequency $\lambda$ admits another real sequence $B=(b_{k})$ such that for every $n$ there are finitely many unique rationals $q_{k}^{n}$ for which
$\lambda_{n}=\sum q^{n}_{k} b_{k}.$
The matrix $R=(q^{n}_{k})$ we call Bohr matrix, and we write $\lambda=(R,B)$, whenever $\lambda$ decomposes with respect to $B$ with associated Bohr matrix $R$.
Given a formal Dirichlet series $D=\sum a_{n}e^{-\lambda_{n}s}$ and a decomposition $\lambda=(R,B)$, the Nth abschnitt $D|_{N}$ of $D$ is the sum $\sum a_{n}e^{-\lambda_{n}s}$, where only those $a_{n}$ differ from $0$ for which $\lambda_{n}$ is a linear combination of the first $b_{1}, \ldots, b_{N}$.

\begin{Theo} \label{Nabschnittvectorvalued} Let $\lambda$ be any frequency with decomposition $\lambda=(R,B)$ and $X$ a Banach space. Let $1\le p \le \infty$ and
$D=\sum a_{n} e^{-\lambda_{n}s} \in \mathcal{D}(\lambda, X)$. Then the following are equivalent:
\begin{enumerate}
\item[(1)]$D \in \Hcal^{+}_{p}(\lambda,X)$,
\item[(2)]$D|_{N} \in \Hcal^{+}_{p}(\lambda,X)$ for all $N$ and $\sup_{N \in \N } \| D|_{N}\|^{+}_{p}<\infty$.
\end{enumerate}
Moreover in this case
\begin{equation} \label{formulaNabschnittvectorvalued}
\|D\|_{p}^{+}=\sup_{N \in \N } \| D|_{N}\|^{+}_{p}<\infty.
\end{equation}
\end{Theo}
We mention that in \cite[Theorem~5.9]{DefantSchoolmann4} this theorem for $X=\mathbb{C}$ is proven under the assumption that Bohr's theorem  holds for $\lambda$. We now present a proof that avoids this assumption.
\begin{proof}[Proof of Theorem \ref{Nabschnittvectorvalued}] First assume that (1) holds. Then, following the argument from \cite[Remark 4.21]{DefantSchoolmann2} together with Remark \ref{convolutionvectorvalued}, we obtain $D|_{N}\in \mathcal{H}_{p}^{+}(\lambda,X)$ with $\|D|_{N}\|^+_{p}\le \|D\|^+_{p}.$ Suppose that (2) holds, and let $\mu_{x}$ be the measure from \eqref{measureRieszmeanvectorvalued} with $x >0$ and $k=1$. Moreover, let $T|_{N}\in \Pi_{cone}^{\lambda}(E_{q}(G),X)$ correspond to $D|_{N}$ (Theorem \ref{conesummingcoincidencevectorvalued}). Then for fixed $x$ and $N$ large enough we have by Remark~ \ref{sylviavectorvalued} and Remark~\ref{convolutionvectorvalued}
$$\Big\|\sum_{\lambda_{n}<x} a_{n} \big(1-\frac{\lambda_{n}}{x}\big)e^{-\lambda_{n}s}\Big\|_{p}=\pi_{cone}(T|_{N}*\mu_{x})\le \sup_{M} \pi_{cone}(T|_{M}),$$
which is finite by assumption and Theorem \ref{conesummingcoincidencevectorvalued}.
Hence Proposition~\ref{tool1vectorvalued} implies $D \in \mathcal{H}_{p}^{+}(\lambda,X)$. It remains to check (\ref{formulaNabschnittvectorvalued}). So let  $T\in \Pi_{cone}^{\lambda}(E_{q}(G),X)$ correspond to $D$ (Theorem \ref{conesummingcoincidencevectorvalued}). Then for every choice of finitely many polynomials $g_{k}\in E_{q}(G)$ we have for large $M$
$$\sum_{k=1}^{n} \|T(g_{k})\|_{X}=\sum_{k=1}^{n} \|T|_{M}(g_{k})\|_{X}\le \sup_{N} \pi_{cone}(T|_{N}) \big \| \sum_{k=1}^{n} |g_{k}|\big\|_{q},$$
which by density shows that $\|D\|_{p}^{+}\le \sup_{N}\|D|_{N}\|_{p}^{+}$.
\end{proof}
\subsection{Coincidence}

Recall that by Corollary \ref{inclusionplusvectorvalued}  for every $1\le p \le \infty$ and  every frequency $\lambda$ we  have the isometric inclusion $\mathcal{H}_{p}(\lambda,X)\subset \mathcal{H}_{p}^{+}(\lambda,X)$.
However as stated in \eqref{arnp2vectorvalued}, for the vector-valued case with $\lambda=(\log n)$ equality is attained if and only if $X$ has ARNP. Likewise, the same is true for the frequency $\lambda=(n)$. As mentioned before, we have that $X$ has ARNP if and only if $H_p(\T,X)=H_p(\D,X)$ which we can rewrite as $\Hcal_p((n),X)=\Hcal^+_p((n),X)$.
 Indeed, as the following theorem proves, ARNP is sufficient for every frequency $\lambda$.

\begin{Theo} \label{mainresultARNPvectorvalued}
Let $\lambda$ be an arbitrary frequency and $X$ a Banach space with ARNP. Then for all $1\le p \le \infty$ we have  $$\mathcal{H}_{p}^{+}(\lambda,X)=\Hcal_{p}(\lambda,X).$$
\end{Theo}

The proof needs some preparation, and is given at the end of this section.
Recall that a Banach space $X$ with ARNP never contains an isomorphic copy of $c_0$. We also point out that Banach lattices $X$ have this property if and only if $c_0$ is not isomorphically contained in $X$.
This is actually a necessary condition for the coincidence result.

\begin{Rema} \label{cnullvectorvalued} If $\Hcal_{p}(\lambda,X)= \Hcal^{+}_{p}(\lambda,X)$ for some $1\le p \le \infty$, then $X$ contains no isomorphic copy of  $c_{0}$.
\end{Rema}
\begin{proof} We assume by contradiction that  there is an isomorphic embedding $c_{0}\hookrightarrow X$,
	and let moreover be a $\lambda$-Dirichlet group $(G, \beta)$ such that $\mathcal{H}_p(\lambda,X)= H_p^\lambda(G, X)$.   Define $f^{N}=\sum_{n=1}^{N} e_{n}h_{\lambda_{n}}$, where the
	$e_{n}$ denote the unit vectors in $c_{0}$. Then for every $\sigma>0$ and $M >N $ we have
	$\|f^{M}_{\sigma}-f^{N}_{\sigma}\|_p=e^{-\lambda_{N+1}\sigma}$. Consequently, $(f^{N}_{\sigma})$ is Cauchy in $H_p^{\lambda}(G,X)$ with limit, say $f_{\sigma}$, and
	$$\|f_{\sigma}\|_p=\lim_{N\to \infty} \|f^{N}_{\sigma}\|_p\le 1.$$
	Hence $D=\sum e_{n}e^{-\lambda_{n}s} \in \mathcal{H}_p^{+}(\lambda,X)$, and so $D\in \mathcal{H}_{p}(\lambda,X)= H_p^\lambda(G, X)$. Consequently, $(e_{n})$  by the Riemann-Lebesgue lemma  is a zero sequence in $X$ and so in $c_{0}$, a contradiction.
\end{proof}

Our proof of Theorem \ref{mainresultARNPvectorvalued} follows the same strategy as in \cite{AntonioDefant}. As an important ingredient we use Lemma 5.4 from \cite{AntonioDefant} which states that  for every
bounded and holomorphic function $F \colon [Re >0] \to X$ the horizontal limit
\begin{equation} \label{ARNPlimitvectorvalued}
\lim_{\varepsilon\to 0} F(\varepsilon+it)
\end{equation}
exists for Lebesgue-almost all $t \in \R$, whenever $X$ has ARNP (see also \cite[Lemma~11.22]{Defant}).
The following lemma is an analogue of  \cite[Proposition 2.3]{AntonioDefant} (also proved in  \cite[Proposition 11.20]{Defant}).
\begin{Lemm} \label{blavectorvalued} For $1 \leq p <\infty$ and $D \in \Hcal_{p}^{+}(\lambda,X)$ the function
$$F \colon [Re> 0] \to \Hcal_{p}(\lambda,X), ~~~ z \mapsto D_{z}=\sum a_{n}(D)e^{-\lambda_{n}z} e^{-\lambda_{n}s}$$
is defined, continuous on $[Re \ge 0]$ and holomorphic on $[Re>0]$. Moreover,
\begin{enumerate}
\item[(1)] $\|D_{\varepsilon+it}\|_{p}=\|D_{\varepsilon}\|_{p}$ for all $t \in \R$ and $\varepsilon\ge 0$,
\item[(2)]  $\sup_{Re \ge 0} \|D_z\|_p = \|D\|^+_p$,
\item[(3)] the function $\varepsilon \to \|D_{\varepsilon}\|_{p}$ is decreasing,
\item[(4)] if $D \in \mathcal{H}_{p}(\lambda,X)$, then $\lim_{\varepsilon \to 0} D_{\varepsilon}=D$ in  $\Hcal_{p}(\lambda,X)$.
\end{enumerate}
\end{Lemm}
For the proof of  (1) take a $\lambda$-Dirichlet group $(G,\lambda)$, and note that by the translation invariance of the Haar measure $m$ on $G$ for every polynomial $D=\sum_{n=1}^{N} a_{n} e^{-\lambda_{n}s}$ we have
$$\|D_{it}\|_{p}^{p}= \int_{G} \big\| \sum_{n=1}^{N} a_{n} e^{-it\lambda_{n}} h_{\lambda_{n}}(\omega) \big\|^{p} d\omega =\int_{G} \big\| \sum_{n=1}^{N} a_{n}  h_{\lambda_{n}}(\omega \beta(t)) \big\|^{p} d\omega=\|D\|_{p}^{p}.$$
Now all claims follow  the same lines as in the ordinary case.

The third ingredient for the proof of Theorem \ref{mainresultARNPvectorvalued} reduces the case $p=\infty$ to finite $p$'s.
\begin{Lemm}\label{coincidenceIIvectorvalued} Let $X$ be a Banach space and $\lambda$ a frequency. If $\Hcal_{p}^{+}(\lambda,X)=\Hcal_{p}(\lambda,X)$ for some $1\le p < \infty$, then $\Hcal_{\infty}^{+}(\lambda,X)=\Hcal_{\infty}(\lambda,X)$.
\end{Lemm}
\begin{proof} By Corollary~\ref{inclusionplusvectorvalued} we only have to check one inclusion, so let $D\in \mathcal{H}_{\infty}^{+}(\lambda,X)$. Then $D \in \mathcal{H}^{+}_{p}(\lambda,X)=\mathcal{H}_{p}(\lambda,X)$, and so there is $f\in H_{p}^{\lambda}(G,X)$ such that $\widehat{f}(h_{\lambda_{n}})=a_{n}(D)$ for all $n$ (for an appropriate $\lambda$-Dirichlet group $(G,\beta)$, as e.g. the Bohr compactification). We show that in fact  $f \in H_{\infty}^{\lambda}(G,X)$. Indeed,
by Proposition~\ref{Hpplusscalarvectorvalued} we know that
$x^{\ast} \circ D \in \mathcal{H}_{\infty}(\lambda)$ with $\|x^{*} \circ D\|_{\infty}\le \|D\|^{+}_{\infty}$.
Hence, comparing Fourier- and Dirichlet coefficients  we see that $x^{*}\circ f\in H_{\infty}^{\lambda}(G)$ for every $x^{*}\in X^{*}$ with $\|x^{*}\circ f\|_{\infty}\le \|D\|^{+}_{\infty}$. Now recall that, since $f$ is measurable, $f$ is separable-valued, i.e. there is a zero set $E\subset G$ and separable subspace $X_{0}$ such that $f(G\setminus E)\subset X_{0}$. An application of the Hahn-Banach theorem shows that there is a sequence  $(x_{n}^{*}) \subset X_{0}^{*}$ such that
\begin{equation} \label{separableHahnBanachvectorvalued}
\|x_{0}\|=\sup_{n} |x^{*}_{n}(x_{0})|
\end{equation} for every $x_{0} \in X_{0}$.
Hence, to every $n$ there is a zero set $E_{N} \subset G$ such that $|x^{*}_{n}(f(\omega))|\le \|D\|_{\infty}^{+}$ for every $\omega \notin E_{N} \cup E$. Collecting countable zero sets we by (\ref{separableHahnBanachvectorvalued})  obtain that $\|f\|_{\infty}\le \|D\|_{\infty}^{+}$.
\end{proof}

\begin{proof}[Proof of Theorem \ref{mainresultARNPvectorvalued}] By Lemma \ref{coincidenceIIvectorvalued} it suffices to prove the claim for $1\le p <\infty$.
 Choose a $\lambda$-Dirichlet group $(G, \beta)$ such that $\mathcal{H}_{p}(\lambda,X)= H_p^\lambda(G,X)$. Since by assumption $X$ has ARNP, the space $L^{p}(G,X)$ has ARNP, and so also $\mathcal{H}_{p}(\lambda,X)$ as a closed isometric subspace (see e.g. \cite[Proposition 23.20]{Defant}). Now, given $D\in \mathcal{H}^{+}_{p}(\lambda,X)$,  using Lemma \ref{blavectorvalued} and (\ref{ARNPlimitvectorvalued}) there is some  $t \in \R$, such that $\lim_{\varepsilon \to 0} D_{\varepsilon+it}$ exists in $\Hcal_{p}(\lambda,X)$. Comparing  Dirichlet coefficients we see  that this limit is exactly $D_{it}$, and so by translation invariance $D \in \Hcal_{p}(\lambda,X)$.
\end{proof}

\subsection{$\mathbb{Q}$-linear independence I}
Theorem \ref{mainresultARNPvectorvalued} and Remark \ref{cnullvectorvalued} show that for a given frequency $\lambda$ and a Banach space $X$ we have
\[X \text{ has ARNP} \Rightarrow \Hcal_p(\lambda,X)=\Hcal^+_p(\lambda,X) \text{ for all } 1\le p \le \infty \Rightarrow c_0\not\subset X.\]
As mentioned before, for the case of classical Fourier series $\lambda=(n)$ and ordinary Dirichlet series $\lambda=(\log n)$, it is known that $X$ has ARNP if and only if  $\Hcal_p(\lambda,X)=\Hcal^+_p(\lambda,X)$. To complete the picture we show that for $\Q$-linearly independent frequencies $\lambda$ we have $\Hcal_p(\lambda,X)=\Hcal^+_p(\lambda,X)$ if and only if $c_0$ is not isomorphic to a subspace of $X$. To this end we identify  $\Hcal_p(\lambda,X)$ and $\Hcal^+_p(\lambda,X)$ with two sequence spaces of random convergence.
For a Banach space $X$ we define
\[
\rad(X) =\left\{(x_n)_{n\in\N}\subseteq X : \ \sum_{n=1}^\infty  x_n z_n \text{ converges almost surely in } X \right\}
\]
and
\[
\Rad(X) =\left\{(x_n)_{n\in\N}\subseteq X : \ \sup_{N\in\N} \Big\|\sum_{n=1}^N x_n z_n\Big\|_{L_2(\T^\infty,X)}<\infty\right\}.
\]
A straightforward computation shows that $\Rad(X)$ is a Banach space under the norm provided in its definition. On the other hand, in \cite[Theorem 3.1 (b)]{vakhania} it is shown that $\rad(X)$ can be identified with the closure in $L_2(\T^\infty,X)$ of
\[\spa\left\{\sum_{n=1}^N x_n z_n: \ N\in\N,\ x_n\in X\right\}.\]
This endows $\rad(X)$ with a Banach space structure.

As a consequence of \cite[Theorem 6.1]{vakhania} we get that $\rad(X)=\Rad(X)$ if and only if $c_0$ is not isomorphic to a subspace of $X$ (apply the equivalence $(a)\Leftrightarrow (c)$
from \cite[Theorem 6.1]{vakhania}
for $\xi_k:\T^\infty\rightarrow L_2(\T^\infty,X)$ given by $\xi_k(w)=w_k(z_k x_k)$).
This together with the characterization provided in the following Theorem settles the issue.

\begin{Theo} \label{radHpcoincidevectorvalued} Let $\lambda$ be $\mathbb{Q}$-linearly independent and $X$ a Banach space. Then for all $1\le p<\infty$ isomorphically
$$\rad(X)\simeq \mathcal{H}_{p}(\lambda,X)\subset \mathcal{H}_{p}^{+}(\lambda,X)\simeq\Rad(X).$$
In particular for every $1\le p \le \infty$ we have $\mathcal{H}_{p}(\lambda,X)= \mathcal{H}_{p}^{+}(\lambda,X)$ if and only if $c_{0}$ is not isomorphic to a subspace of $X$.
\end{Theo}

We will need the Kahane-Khintchine inequality (see e.g. \cite[(14.18) and Theorem 25.9 ]{Defant}) to compare the $p$-th moments of random sums: For every $1\leq p<\infty$ there are constants $A_p,B_p>0$ such that for every Banach space $X$ and every choice of vectors $(x_n)_{n=1}^N\subseteq X$ we have
\begin{equation}\label{Kahaneinequalityvectorvalued}
A_P^{-1}\Big\|\sum_{n=1}^N x_n z_n\Big\|_{L_p(\T^\infty,X)}
\leq \Big\|\sum_{n=1}^N x_n z_n\Big\|_{L_2(\T^\infty,X)}
\leq B_p \Big\|\sum_{n=1}^N x_n z_n\Big\|_{L_p(\T^\infty,X)}.
\end{equation}

\begin{proof}[Proof of Theorem \ref{radHpcoincidevectorvalued}]
Notice that $\T^\infty$ (with the canonical morphism $\beta$) is a Dirichlet group for $\lambda$. As a direct consequence of the Kahane-Khintchine inequality (\ref{Kahaneinequalityvectorvalued}) and the density of finite sums $\sum_{n=1}^N x_n z_n$ in $\rad(X)$ and $H_p^\lambda(\T^\infty,X)$ we have that
\[\rad(X)\simeq H_p^\lambda(\T^\infty,X)=\mathcal{H}_{p}(\lambda,X).\]
For the second equality notice that using Corollary \ref{inclusionplusvectorvalued} and again the Kahane-Khintchine inequality for any sequence $(x_n)_n\subseteq X$ we have
\begin{align*}
   \|\sum_{n=1}^N x_n z_n\|_{\rad(X)} &=\sup_{N\in\N}\Big\|\sum_{n=1}^N x_n z_n\Big\|_2
   =\sup_{N\in\N}\sup_{\sigma>0}\Big\|\sum_{n=1}^N x_n e^{-\lambda_n\sigma}z_n\Big\|_2
    \\ &\simeq \sup_{N\in\N}\sup_{\sigma>0}\Big\|\sum_{n=1}^N x_n e^{-\lambda_n\sigma}z_n\Big\|_p
    =  \sup_{N\in\N} \|D|_N\|_{\mathcal{H}_{p}^+(\lambda,X)}.
\end{align*}
This together with Theorem \ref{Nabschnittvectorvalued} proves the equality.

It only remains to check the case $p=\infty$ of the last statement. This is a direct consequence from Lemma \ref{coincidenceIIvectorvalued} and Remark \ref{cnullvectorvalued}.
\end{proof}

\subsection{Brothers Riesz theorem}

Let us discuss the special case $p=1$ in Theorem~\ref{conesummingcoincidencevectorvalued}: For every $\lambda$-Dirichlet group  $(G, \beta)$  and every Banach space $X$
\[
 \Pi^{\lambda}_{cone}(C(G),X) = \Hcal^{+}_{1}(\lambda,X), ~~ T \mapsto \sum \widehat{T}(h_{\lambda_{n}}) e^{-\lambda_{n}s}
\]
 is an onto isometry. Denote by $M(G,X)$ the Banach space of all regular   $X$-valued Borel measures on $G$ of bounded variation, and
 by $M_{\lambda}(G,X)$ its subspace of $\lambda$-analytic measures $\mu$, i.e $\widehat{\mu}(\gamma)\neq 0$ only if $\gamma = h_{\lambda_n}$
 for some $n$.  A well-known result (see e.g. \cite[Chapter VI]{DiestelUhl})
 shows that isometrically
 \[
  M(G,X) = \Pi_1(C(G), X)\,, \,\, \mu \mapsto \Big[ g \mapsto \int_G g ~d\mu \Big]\,,
 \]
 where as above $\Pi_1$ denotes the summing operators. As a consequence of what we have achieved so far, we  state the following Brother Riesz type theorem for general $X$-valued
 $\lambda$-Dirichlet series.

 \begin{Theo} Let $\lambda$ be arbitrary, $(G,\beta)$ a $\lambda$-Dirichlet group, and $X$ a Banach space. Then the mapping
$$ M_{\lambda}(G,X) = \Hcal^{+}_{1}(\lambda,X), ~~ \mu \mapsto \sum \widehat{\mu}(h_{\lambda_{n}})e^{-\lambda_{n}s},$$ defines an onto isometry
preserving the Fourier- and Dirichlet coefficients.  Moreover,
\begin{enumerate}
\item[(1)]  $M_{\lambda}(G,X)=\mathcal{H}_{1}(\lambda,X)$ whenever $X$ has ARNP,
\end{enumerate}
and under the assumption that $\lambda$ is $\mathbb{Q}$-linearly independent,
\begin{enumerate}
\item[(2)] $M_{\lambda}(G,X)=\mathcal{H}_{1}(\lambda,X)$ if and only if $X$ contains no isomorphic copy of $c_0$.
\end{enumerate}
\end{Theo}
Note that \eqref{cone2vectorvalued} and Theorem \ref{conesummingcoincidencevectorvalued} immediately give the first result.
Then both  statements (1) and (2) follow from Theorem~\ref{mainresultARNPvectorvalued} and Theorem~\ref{radHpcoincidevectorvalued}.
The scalar ordinary case is due to \cite{HelsonLowdenslager} (for an alternative proof see also \cite{AlemanOlsenSaksman}).
In the ordinary vector-valued case (1) was proved in \cite{AntonioDefant} (see also \cite[Section 26.6]{Defant}).  The identity $M_{\lambda}(G)=\mathcal{H}_{1}(\lambda)$,   where $X =\C$ and $\lambda$ is arbitrary, is done in \cite[Theorem 4.25]{DefantSchoolmann2}; there the proof  uses as a crucial ingredient a result of  Doss from \cite[Theorem 4]{Doss}  for locally compact and connected groups with ordered duals. Note that the proof given here   may be seen as a proof  which is entirely performed within  Dirichlet series.

 \subsection{Preduals of $\mathcal{H}_{p}(\lambda,X^{*})$}
As e.g. proved in \cite[Proposition 24.16]{Defant} for every $1\le p \le \infty$ and every  Banach space $X$ we have that isometrically
\begin{equation} \label{marshallvectorvalued}
\Pi_{cone}(E_{q}(G),X^{*})=E_{q}(G,X)^{*}\,, \,\, \,[f \otimes x \mapsto < Tf,x>].
\end{equation}
Hence, fixing a frequency $\lambda$ and denoting by $E_{q}(G,X)^{*}_{\lambda}$ the weak*-closed subspace of all functionals $\varphi \in E_{q}(G,X)^{*}$ such that $\varphi(\overline{h_{t}}\otimes x)=0$ for all $x\in X$, whenever for  $t \notin \{\lambda_{n} \mid n \in \N\}$, we by  (\ref{marshallvectorvalued}) and Theorem \ref{conesummingcoincidencevectorvalued} have the isometric equalities
$$\mathcal{H}_{p}^{+}(\lambda,X^{*})=\Pi^{\lambda}_{cone}(E_{q}(G),X^{*})=E_{q}(G,X)^{*}_{\lambda},$$
where again the Dirichlet- and Fourier coefficients are preserved. In the ordinary case the first part of the  following theorem was proved in \cite[Theorem~7.3]{AntonioDefant}
(compare also  \cite[Theorem 24.15]{Defant}). The proof follows similar lines as in the ordinary case.

\begin{Theo}\label{predualvectorvalued} Let $\lambda$ be a frequency, $X$  a Banach space and $1\le p \le \infty$. Then $\Hcal_{p}(\lambda,X^{*})$ has a predual if and only if $$\Hcal_{p}^{+}(\lambda,X^{*})=\Hcal_{p}(\lambda,X^{*})\,.$$
In particular, each of these equivalent statements holds
\begin{enumerate}
\item[(1)] whenever $X^*$ has ARNP,
\end{enumerate}
and,  assuming that $\lambda$ is $\mathbb{Q}$-linearly independent,
\begin{enumerate}
\item[(2)] if and only if $X^{*}$ does not contain an isomorphic copy of $c_{0}$.
\end{enumerate}
\end{Theo}
In contrast to the ordinary case, there are Banach spaces $X$ such that $\Hcal_{p}(\lambda,X^{*})$ has a predual but $X^*$ fails ARNP. Indeed, it suffices to find a dual space failing ARNP that contains no isomorphic copy of $c_0$, since then  for $\Q$-linearly independent frequencies \textit{(2)} holds.

We show that if $A$ is the disk algebra, then $A^*$ satisfies the desired properties.
It does not contain $c_0$, since $A^*$ has cotype 2 (see \cite[Corollary~2.11]{Bou84}).
We now check that $A^*$ fails ARNP.
As shown in \cite[page 11]{Pel}, the F. and M. Riesz theorem  implies that
\[A^*\simeq L_1(\T)/H_0^1(\T) \oplus_1 M_s(\T),\]
where $M_s(\T)$ denotes the Banach space of singular measures on $\T$. Since the quotient $L_1(\T)/H_0^1(\T)$ does not have ARNP (see \cite[Remark~4.35]{Pi16}) and this property is inherited by closed subspaces, we deduce that $A^*$ also fails ARNP. In conclusion, for $\Q$-linearly independent frequencies $\lambda$ we have that  $\Hcal_{p}(\lambda,A^{*})$ has a predual by Theorem \ref{predualvectorvalued} but $A^*$ fails ARNP.

\section{General Dirichlet series vs. holomorphic functions}\label{sec4vectorvalued}

The main goal of this section is to extend the equivalences of Bohr's theorem to the vector-valued setting (see Aspect III). Surprisingly, this depends only on the frequency $\lambda$ and not on the geometry of the Banach space $X$ once we replace $\HH_\infty(\lambda,X)$ with $\HH^+_\infty(\lambda,X)$ as mentioned in the introduction.

\subsection{Besicovitch spaces} Recall that the Banach space $\mathcal{H}_{\infty}^{\lambda}([Re>0],X)$
consists  of all bounded and holomorphic functions $F\colon [Re>0] \to X$, which are  almost periodic on all abscissas and for which the Bohr coefficients $a_{x}(F)$ vanish whenever $x \notin \{\lambda_{n} \mid n \in \N\}$.

In the scalar case $X=\mathbb{C}$ we know from \cite[Theorem 2.16]{DefantSchoolmann3} that there is an onto isometry
\begin{equation} \label{Besiscalarcasevectorvalued}
\mathcal{H}_{\infty}^{\lambda}[Re>0] = \mathcal{H}_{\infty}(\lambda)
\end{equation}
preserving the Bohr and Dirichlet coefficients. We extend this result to the $X$-valued case.

\begin{Theo} \label{Besivectorvalued}
For all frequencies $\lambda$  and Banach spaces $X$ there is  an onto isometry
\[
\mathcal{H}_\infty^\lambda([Re >0], X) =  \mathcal{H}^+_\infty(\lambda, X), ~~ F \mapsto D, \,
\]
such that $a_{\lambda_n}(F)= a_n(D)$ for all $n$. In particular, the inclusion
$$   \mathcal{D}_{\infty}(\lambda, X)  \hookrightarrow \mathcal{H}^+_\infty(\lambda, X)\,.$$
is isometric.
\end{Theo}

Before going into the proof, from Corollary~\ref{heutevectorvalued} we easily deduce the following approximation theorem for almost periodic functions.

\begin{Coro} \label{fireinthemorningvectorvalued} Let $\lambda$ be a frequency,  $X$ a Banach space, and $k >0$. Then for every $F\in \mathcal{H}_\infty^\lambda([Re >0], X)$ and $\sigma>0$ the restriction
$$F_{\sigma}: \R \to X, \, F_\sigma(t) = F (\sigma +it)$$
is,  as $x\to \infty$, the uniform limit of
the polynomials
$$\sum_{\lambda_{n}<x} a_{\lambda_{n}}(F)\big(1-\frac{\lambda_{n}}{x}\big)^{k} e^{-\lambda_{n}(\sigma +it)}\,.$$
\end{Coro}

In order to prove Theorem \ref{Besivectorvalued} we first extend Theorem~\ref{HahnBanachinftyvectorvalued} from $\mathcal{D}_{\infty}(\lambda,X)$ to $\mathcal{H}_{\infty}^{+}(\lambda,X)$.

\begin{Theo}\label{weakvectorvalued}
Let $D \in \mathcal{D}(\lambda, X)$. Then the following are equivalent:
\begin{itemize}
\item[(1)]
$D\in \mathcal{H}^+_\infty(\lambda, X)$
\item[(2)]
$x^\ast \circ D\in \mathcal{H}_\infty(\lambda)$ for all $x^\ast \in X^\ast$
\end{itemize}
Moreover, in this case
\begin{equation} \label{HahnBanachHpluesvectorvaled}
\|D\|^{+}_\infty = \sup_{x^\ast  \in B_{X^\ast}} \|x^\ast \circ D\|_{\mathcal{H}_\infty(\lambda)}.
\end{equation}
\end{Theo}

\begin{proof}
 Suppose that $D\in \mathcal{H}^+_\infty(\lambda, X)$. Then $x^{*} \circ D \in \mathcal{H}_{\infty}^{+}(\lambda)$ with $\|x^{*}\circ D\|_{\infty}^{+}\le \|D\|_{\infty}^{+}$ for every $x^{*}$ with $\|x^{*}\|=1$. Now the claim follows, since  $\mathcal{H}_{\infty}^{+}(\lambda)=\mathcal{H}_{\infty}(\lambda)$ by Proposition~\ref{Hpplusscalarvectorvalued}. Assume conversely that $D$ satisfies (2). By a closed graph argument we have that
\[
\sup_{x^\ast \in B_{X^\ast}}\|x^\ast \circ D\|_{\mathcal{H}_\infty(\lambda)} =C < \infty.
\]
Moreover fixing a $\lambda$-Dirichlet group $G$, for every $x^\ast \in B_{X^\ast}$ there is a function $f_{x^{*}} \in \mathcal{H}_\infty^{\lambda}(G)$ such that $\widehat{f_{x^\ast}}(h_{\lambda_n})= x^\ast (a_n)$ for all $n$ and $\|f_{x^\ast}\|_\infty = \|x^\ast \circ D\|_{\mathcal{H}_\infty(\lambda)}$. Now consider the linear operator
$$ T\colon L_{1}(G) \to X^{**}, ~~ g \mapsto  \Big[x^{*}\mapsto \int_{G} g(\omega) f_{x^{*}}(\omega)~ d\omega\Big].$$
Then $T$ is bounded with $\|T\|\le C$ and so $T\in \Pi_{cone}(L_{1}(G),X^{**})$ by (\ref{cone1vectorvalued}). Since for all $x^{*}\in X^{*}$
$$\widehat{T}(h_{x})(x^{*})=T(\overline{h_{x}})(x^{*})=\widehat{f_{x^{*}}}(h_{x})$$
we have $T\in \Pi^{\lambda}_{cone}(L_{1}(G),X^{**})$ with
$$\widehat{T}(h_{\lambda_{n}})(x^{\ast})=x^{\ast}(a_{n}(D))\in X \,\,\,\text{for all $n$}$$
for every $x^{\ast}\in X^{\ast}$. Hence we identify $\widehat{T}(h_{\lambda_{n}})=a_{n}\in X$ so that for every polynomial $g= \sum a_{x_{n}}h_{x_{n}}$ we consequently have $T\in \Pi^{\lambda}_{cone}(L_{1}(G),X)$, which finishes the claim by Theorem \ref{conesummingcoincidencevectorvalued}.
\end{proof}

\begin{proof}[Proof of Theorem \ref{Besivectorvalued}]
Let $F \in \mathcal{H}_\infty^\lambda([Re >0], X)$ and $D=\sum a_{n}e^{-\lambda_{n}s}$ be defined by  $a_{n}=a_{\lambda_n}(F)$. Then $x^{*} \circ F\in \mathcal{H}_\infty^\lambda[Re >0]$, and so by (\ref{Besiscalarcasevectorvalued}) and comparing coefficients we have $x^{*}\circ D\in \mathcal{H}_{\infty}(\lambda)$ for all $x^{*}$. Now Theorem \ref{weakvectorvalued} implies $D\in \mathcal{H}_{\infty}^{+}(\lambda,X).$ Conversely, take some $D = \sum a_n e^{\lambda_n s} \in \mathcal{H}^+_\infty(\lambda, X)$, i.e. for each $\sigma >0$ there is some
$f_\sigma \in H_\infty^\lambda (G,X)$ such that $\widehat{f_{\sigma}}(h_{\lambda_n}) = a_{n} e^{-\sigma \lambda_{n}}$ for all $n$ and $\|f_{\sigma}\|=\|D_{\sigma}\|_{\infty}$.
Using the measures from (\ref{measureRieszmeanvectorvalued}) with $k=1$,  we for all $x>0$ have
\begin{align*}
\sup_{t \in \R} \big| \sum_{\lambda_{n}<x} x^\ast(a_{n})
  &
  e^{-\lambda_n \sigma}
    \big(1- \frac{\lambda_n}{x}\big)
   e^{-it \lambda_n}\big|
   \\
   &
=
\|(x^{*}\circ f_\sigma) \ast \mu_x\|_{\infty} \leq \|f_\sigma\|_{\infty} \|\mu_x\| \leq C_{1} \|D\|_\infty.
\end{align*}
We conclude by Proposition~\ref{BohrCahenIIvectorvalued} that $\sigma_u^{\lambda,1}(D) \leq 0$,
so that the function
\[
F: [Re >0] \to X, \, F(s) := \lim_{x \to \infty}
\sum_{\lambda_{n}<x} a_{n}
      \big(1- \frac{\lambda_n}{x}\big)
   e^{-\lambda_n s}
\]
belongs to $\mathcal{H}_{\infty}^{\lambda}([Re>0],X)$ and has $a_n$ as its $\lambda_{n}$th Bohr coefficient. Indeed, this function  is bounded, since for every $\sigma >0$
\begin{align*}
\sup_{t \in \R}\|F(\sigma+it)\|_X
&
=
\lim_{x\to\infty} \sup_{t \in  \R} \big\|\sum_{\lambda_{n}<x} a_{n}\big(1-\frac{\lambda_{n}}{x}\big) e^{-\sigma\lambda_n} e^{-i\lambda_n t}\big\|_X
\\&
= \lim_{x \to \infty} \|f_\sigma \ast \mu_x\|_{\infty}
\leq
 \|f_\sigma\|_{\infty} C_{1} \le C_{1} \|D\|^{+}_{\infty}\,,
\end{align*}
which finishes the proof.
\end{proof}

\subsection{Bohr's theorem} \label{Bohr's theoremvectorvalued}

Suppose that $D=\sum a_{n}e^{-\lambda_{n}s}$ converges somewhere and that its limit function extends  to a bounded and holomorphic function $f$ on $[Re>0]$. Then as already mentioned in Aspect III from the introduction a prominent problem from the beginning of the 20th century was to determine the class of $\lambda$'s for which under this assumption all $\lambda$-Dirichlet series converge uniformly on $[Re>\varepsilon]$ for every $\varepsilon>0$.
 We say that $\lambda$ satisfies 'Bohr's theorem'  if the answer to the preceding problem is affirmative. In general the answer is negative. See \cite[Theorem 5.2]{Schoolmann} for examples of $\lambda$'s which fail for Bohr's theorem.

Considering $X$-valued Dirichlet series  one may ask if it does make sense to define  '$\lambda$ satisfies Bohr's theorem for the Banach space $X$'
 whenever every Dirichlet series $D=\sum a_{n}e^{-\lambda_{n}s}$  with coefficients in $X$ which
 converges somewhere and has a limit function extending  to a bounded and holomorphic function $f$ on $[Re>0]$ with values in $X$,
  converges uniformly on $[Re>\varepsilon]$ for every $\varepsilon>0$. In this context the space $\mathcal{D}_{\infty}^{ext}(\lambda,X)$ of all somewhere convergent $D\in \mathcal{D}(\lambda,X)$, that allow a holomorphic and bounded extension $f$ to $[Re>0]$ is natural. Actually, as a consequence of Proposition \ref{BohrCahenIIvectorvalued} we see that the Banach space $X$ does not affect satisfying Bohr's theorem.

\begin{Prop} \label{mainresultingredient1vectorvalued} Let $\lambda$ be a frequency and $X$  a non-trivial Banach space. Then
 $\lambda$  satisfies Bohr's theorem if and only if  $\lambda$  satisfies Bohr's theorem for $X$.
\end{Prop}

\begin{proof}
Assume that $\lambda$  satisfies Bohr's theorem, and let $D \in \mathcal{D}_{\infty}^{ext}(\lambda,X)$.  Then by assumption $\sigma_u(x^\ast \circ D) \leq 0$
for every $x^\ast \in X^\ast$, which by Proposition~\ref{BohrCahenIIvectorvalued} implies that as desired $\sigma_u(D) \leq 0$.
The converse implication is trivial.
\end{proof}

Given a frequency $\lambda$ and a Banach space $X$,  to obtain quantitative versions of Bohr's theorem means to estimate the norm of the partial sum operator
$$S_{N}\colon \mathcal{D}_{\infty}^{ext}(\lambda,X) \to \mathcal{D}_{\infty}(\lambda,X), ~~ D=\sum a_{n}e^{-\lambda_{n}s} \mapsto \sum_{n=1}^{N} a_{n}e^{-\lambda_{n}s}.$$
For the scalar case $X=\mathbb{C}$ we deduce from \cite[Theorem 3.2]{Schoolmann} the following estimate  which does not assume any condition on $\lambda$:
For every $D=\sum a_{n}e^{-\lambda_{n}s} \in \mathcal{H}_{\infty}(\lambda)$ and $0<k\le 1$ we have
\begin{equation*}
\big\|\sum_{n=1}^{N} a_{n}(D) e^{-\lambda_{n}s} \big\|_{\infty} \le \frac{C}{k} \Big(\frac{\lambda_{N+1}}{\lambda_{N+1}-\lambda_{N}}\Big)^{k} \|D\|_{\infty}
\end{equation*}
Applying the Hahn-Banach theorem and (\ref{HahnBanachHpluesvectorvaled}) this result extends to $\mathcal{H}_{\infty}^{+}(\lambda,X)$ .
\begin{Theo}\label{quantitativevectorvalued}  Let $\lambda$ be an arbitrary frequency and $X$ a Banach space. Then for all $D\in \mathcal{H}^{+}_{\infty}(\lambda, X)$, $0<k\le 1$ and $N$ we have
\begin{equation*}
\big\|\sum_{n=1}^{N} a_{n}(D) e^{-\lambda_{n}s} \big\|_{\infty} \le \frac{C}{k} \Big(\frac{\lambda_{N+1}}{\lambda_{N+1}-\lambda_{N}}\Big)^{k} \|D\|_{\infty}^{+},
\end{equation*}
where $C>0$ is a universal constant.
\end{Theo}

Several sufficient conditions on $\lambda$ that guarantee
Bohr's theorem are known. Initially, Bohr in \cite{Bohr} introduces the following condition, which we call Bohr's condition $(BC)$:
 \begin{equation*}
 \exists ~l >0 ~ \forall ~\delta >0 ~\exists ~C>0~\forall   ~n \in \N\colon ~~\lambda_{n+1}-\lambda_{n}\ge Ce^{-(l+\delta)\lambda_{n}}.
\end{equation*}
Note that $\lambda=(\log n)$ has $(BC)$ with $l=1$.
Secondly there is a weaker condition than $(BC)$, namely Landau's condition $(LC)$ from \cite{Landau}:
$$\forall ~\delta>0 ~\exists ~C>0~ \forall ~n \in \N \colon ~~ \lambda_{n+1}-\lambda_{n}\ge C e^{-e^{\delta\lambda_{n}}}.$$
To see an example that has (LC) and fails for (BC) take $\lambda=(\sqrt{\log n})$.
Assuming (LC), the choice $k_N=e^{-\delta\lambda_{N}}$ in Theorem \ref{quantitativevectorvalued}   leads to
$$\|S_{N}\|_{\infty}\le Ce^{\delta \lambda_{N}},$$
which in fact is the vector-valued quantitative variant of Bohr theorem under (LC). Assuming  (BC) the  choice $k_{N}=\lambda_{N}^{-1}$ (here $N\ge 2$, since $\lambda_{1}=0$ is possible) yields
$$\|S_{N}\|\le C \lambda_{N};$$
in the ordinary scalar case $\lambda=(\log n)$ this was first proven in \cite{BalasubramanianCaladoQueffelec} (see also \cite[Theorem 6.2.2]{QQ} and \cite[Theorem 1.13 and (24.14)]{Defant}).

\subsection{Equivalence} \label{Equivalencevectorvalued}

Given a frequency $\lambda$ and a Banach space $X$, we say that $\lambda$ satisfies Bayart's Montel theorem for  $X$ whenever
the following statement holds:
Every  sequence $(D^{N})$ of Dirichlet series $D^N = \sum a_n^Ne^{-\lambda_{n}s} \in  \mathcal{D}_{\infty}(\lambda,X)$ admits a subsequence $(N_{k})$ and $D\in \mathcal{D}_{\infty}(\lambda,X)$ such that $(D^{N_{k}})$ converges to $D$ uniformly on $[Re>\varepsilon]$ for every $\varepsilon>0$ as $k\to \infty$ provided $(D^{N})$ satisfies the following two conditions:
\begin{itemize}
\item[(a)]
There is subsequence $(N_k)_k$ such that $\lim_{k \to \infty} a_n^{N_k}$ exists for all $n$,
\item[(b)]
and $(D^N)$ is bounded in $\mathcal{D}_\infty(\lambda)$\,.
\end{itemize}
 If $X=\C$, then  we shortly say that $\lambda$ satisfies Bayart's Montel theorem. In this case the first assumption $(a)$ on $(D^N)$ by compactness is superfluous,  since we by (b) have that
 $|a_n^N| \leq \sup_N \|D^{N}\|_\infty < \infty$ for all $n$.  In \cite[Lemma 18]{Bayart} Bayart proves that  $\lambda = (\log n)$ has this property.

 Mainly collecting results from  \cite{DefantSchoolmann4} and  \cite{DefantSchoolmann3}, we see that in the case of scalar-valued general Dirichlet series several of the aspects we so far looked at, in fact generate the same classes of  frequencies.
\begin{Theo} \label{scalarcase}
Let $\lambda$ be a frequency. Then the following are equivalent:
\begin{itemize}
\item[(1)] $\lambda$ satisfies Bohr's theorem.
\item[(2)] $\mathcal{D}_{\infty}(\lambda)$ is complete.
\item[(3)] $\mathcal{D}_{\infty}(\lambda)=\mathcal{H}_{\infty}(\lambda)$.
\item[(4)] $\mathcal{D}_{\infty}(\lambda)= \mathcal{H}_{\infty}^{\lambda}[Re>0]$.
\item[(5)] $\lambda$ satisfies Bayart's Montel theorem\,.
\end{itemize}
\end{Theo}
The first four equivalences are known from \cite[Theorem~5.1]{DefantSchoolmann4} and  \cite[Theorem~2.16]{DefantSchoolmann3}, and looking at \cite[Theorem 5.8]{DefantSchoolmann4}
we see that each of them implies $(5)$. We close the cycle including the proof of the implication $(5) \Rightarrow (2)$.

\begin{proof}[Proof of $(5) \Rightarrow (2)$]
We check that (5) implies  that $\mathcal{D}_{\infty}(\lambda)\subset \mathcal{H}_{\infty}^{\lambda}[Re>0]$ is closed, and so (2) follows. Indeed, let $(D^{N})$ be a sequence in $\mathcal{D}_{\infty}(\lambda)$ that converges to a function $F\in \mathcal{H}_{\infty}^{\lambda}[Re>0]$ uniformly on $[Re>0]$. Assuming (5), there is a subsequence $(N_{k})$ and $D\in \mathcal{D}_{\infty}(\lambda)$ such that $D^{N_{k}}$ converges to $D$ on $[Re>\varepsilon]$ for every $\varepsilon>0$ as $k\to +\infty$. This implies that $F=D$ with $a_{\lambda_{n}}(F)=a_{n}(D)$ for all $n$ which finishes the proof.
\end{proof}

The  following vector-valued extension of Theorem~\ref{scalarcase} is the main contribution in this section.

\begin{Theo} \label{maininftyvectorvalued} Let $\lambda$ be a frequency and $X$ a non-trivial Banach space. Then the following are equivalent:
\begin{enumerate}
\item[(1)] $\lambda$ satisfies Bohr's theorem for $X$.
\item[(2)] $\mathcal{D}_{\infty}(\lambda,X)$ is complete.
\item[(3)] $\mathcal{D}_{\infty}(\lambda,X)=\mathcal{H}_{\infty}^{+}(\lambda,X)$.
\item[(4)] $\mathcal{D}_{\infty}(\lambda,X)=\mathcal{H}_{\infty}^{\lambda}([Re>0],X)$.
\item[(5)]
 $\lambda$ satisfies Bayart's Montel theorem for  $X$.
 \end{enumerate}
\end{Theo}
Before turning to the proof of these equivalences, we add another remark.

\begin{Rema} \label{ingovectorvalued}
The equality (4) is equivalent to the fact that every  Dirichlet series $D=\sum a_{\lambda_{n}}(F)e^{-\lambda_{n}s}$ generated by a function  $F\in \mathcal{H}_{\infty}^{\lambda}([Re>0],X)$ converges on $[Re>0]$.
Analogously, (3) holds if and only if every  Dirichlet series $D=\sum a_n e^{-\lambda_{n}s} \in \mathcal{H}_{\infty}^{+}(\lambda,X)$  converges on $[Re>0]$.
\end{Rema}

\begin{proof}[Proof of Remark \ref{ingovectorvalued}]
Clearly, if (4) holds, then every  $D=\sum a_{\lambda_{n}}(F)e^{-\lambda_{n}s}$ generated by some   $F\in \mathcal{H}_{\infty}^{\lambda}([Re>0],X)$ is in $\mathcal{D}_\infty(\lambda, X)$, so converges on $[Re>0]$.
Assume conversely, that  every   $D=\sum a_{\lambda_{n}}(F)e^{-\lambda_{n}s}$ generated by some  $F\in \mathcal{H}_{\infty}^{\lambda}([Re>0],X)$ converges on $[Re>0]$. Then by \cite[Proposition 1.2]{DefantSchoolmann3} (see also \cite[Chapter V]{HardyRiesz}) the Dirichlet series  $D$ is $(\lambda,1)$-summable at every $s\in [Re>0]$, that is
$$D(s)=\lim_{x\to \infty} \sum _{\lambda_{n}<x} a_{\lambda_{n}}\big(1-\frac{\lambda_{n}}{x}\big)e^{-\lambda_{n}s}.$$
By Corollary  \ref{fireinthemorningvectorvalued} this limit coincides with $F(s)$, which implies $D\in \mathcal{D}_{\infty}(\lambda,X)$.
\end{proof}

Starting the proof of Theorem~\ref{maininftyvectorvalued}, we note first that  by Proposition \ref{mainresultingredient1vectorvalued} statement (1) holds if and only if
$\lambda$ satisfies Bohr's theorem.
We will then see that together with Theorem \ref{Besivectorvalued} and Theorem \ref{scalarcase} the proof of Theorem \ref{maininftyvectorvalued} is evident  once we prove the following two results.

\begin{Prop} \label{maininfty2vectorvalued} Let $\lambda$ be a frequency and $X$ a non-trivial Banach space. Then $\mathcal{D}_{\infty}(\lambda)=\mathcal{H}_{\infty}(\lambda)$ if and only if $\mathcal{D}_{\infty}(\lambda,X)=\mathcal{H}^{+}_{\infty}(\lambda,X)$. Moreover, if $X$ has ARNP, then each of these statements is equivalent to
$\mathcal{D}_{\infty}(\lambda,X)=\mathcal{H}_{\infty}(\lambda,X)$.
\end{Prop}

\begin{proof}
This is an immediate consequence of  Theorem \ref{HahnBanachinftyvectorvalued} and Theorem \ref{weakvectorvalued}.
\end{proof}

\begin{Prop} \label{bumbumvectorvalued} Let $\lambda$ be a frequency and $X$ a non-trivial Banach space. Then
 $\mathcal{D}_{\infty}(\lambda)$ is complete if and only if $\mathcal{D}_{\infty}(\lambda,X)$ is complete.
\end{Prop}

\begin{proof}
 Completeness of $\mathcal{D}_{\infty}(\lambda,X)$ implies completeness of $\mathcal{D}_{\infty}(\lambda)$, since the second  space can be viewed as a closed subspace of the first one. Assume
  conversely that $\mathcal{D}_{\infty}(\lambda)$ is complete. If $(D^{N}) \subset \mathcal{D}_{\infty}(\lambda,X)$ is Cauchy, then $(a_n)$ is Cauchy in $X$
  and $(x^{\ast}\circ D^{N})$ is Cauchy in $\mathcal{D}_{\infty}(\lambda)$ for all $x^{\ast} \in X^{\ast}$.
  Define $a_{n}:=\lim_{N\to\infty} a_{n}^{N}, n \in \N$ and $D:=\sum a_{n}e^{-\lambda_{n}s} \in \mathcal{D}(\lambda, X)$.
  Then  $x^\ast\circ D = \lim_N x^\ast \circ D^N \in \mathcal{D}_{\infty}(\lambda)$ with $\|D\|_\infty = \sup_{\|x^{\ast}\|=1} \|x^{\ast}\circ D\|_{\infty} < \infty$, and by
  Theorem \ref{HahnBanachinftyvectorvalued} we see that  $D \in \mathcal{D}_{\infty}(\lambda,X)$.
  It remains to check that $\lim_{N\to \infty} D^{N}=D$ in $\mathcal{D}_\infty(\lambda, X)$.
   Indeed, for  $\varepsilon>0$ take $N_0$ such that  $\|D^{N}-D^{M}\|_{\infty}\le \varepsilon$ for all $M,N\ge N_{0}$.
    Fix now  $x^{\ast}\in X^{\ast}$ with  $\|x^{\ast}\|=1$, and  take $M\ge N_{0}$ such that $\|x^{\ast}(D-D^{M})\|_\infty \le \varepsilon$. Then together for all $N\ge N_{0}$
$$\|x^{\ast}\circ (D-D^{N})\|_\infty \le \|x^{\ast}\circ (D-D^{M})\|_\infty +\|D^{M}-D^{N}\|_\infty \le 2\varepsilon,$$
and so $\|D-D^{N}\|_{\infty}=\sup_{\|x^{\ast}\|=1} \|D-D^{N}\|_{\infty}\le 2\varepsilon$.
\end{proof}

Finally, we collect all partial results for the proof of Theorem~\ref{maininftyvectorvalued}.

\begin{proof}[Proof of Theorem \ref{maininftyvectorvalued}]
Propositions \ref{mainresultingredient1vectorvalued}, \ref{maininfty2vectorvalued} and \ref{bumbumvectorvalued} show the equivalence of (1), (2) and (3), and  Theorem \ref{Besivectorvalued} clearly proves   that  (3) and (4)
are equivalent.
Clearly, if $(5)$ holds, then $\lambda$ satisfies Bayart's Montel theorem (for $\mathbb{C}$), and then we know from  Theorem~\ref{scalarcase} that $\lambda$ satisfies Bohr's theorem (for $\mathbb{C}$), and hence
by Proposition~\ref{mainresultingredient1vectorvalued} also for $X$.
 Finally, we note that the proof of the implication $(1) \Rightarrow (5)$ follows by a word-by-word extension of the proof of \cite[Theorem 5.8]{DefantSchoolmann4}
 (this proof uses a straight forward vector-valued extension of the Bohr-Cahen formula from \cite[Proposition 2.4]{Schoolmann} as well as \cite[Lemma 5.2]{DefantSchoolmann4}).
\end{proof}

Let us again come back to a characterization through $N$th abschnitte like in Theorem~\ref{Nabschnittvectorvalued}, this time for $\mathcal{D}_{\infty}(\lambda,X)$.

\begin{Coro} Let $\lambda$ satisfy Bohr's theorem, $D=\sum a_{n} e^{-\lambda_{n}s}$ be a formal $\lambda$-Dirichlet series and $X$ be any Banach space. Then the following are equivalent:
\begin{enumerate}
\item[(1)]$D \in \mathcal{D}_{\infty}(\lambda,X)$
\item[(2)]$D|_{N} \in \mathcal{D}_{\infty}$ and $\sup_{N \in \N } \| D|_{N}\|_{\infty}<\infty$.
\end{enumerate}
Moreover in this case, $\|D\|_{\infty}=\sup_{N \in \N } \| D|_{N}\|_{\infty}<\infty$.
\end{Coro}
\begin{proof}
Follows immediately from Theorem \ref{Nabschnittvectorvalued} and  \ref{maininftyvectorvalued}.
\end{proof}

We finish with another equivalence on Bohr's theorem -- this time in terms of a concrete inequality.  Speaking in vague terms, one might expect that the question of whether or not a given frequency $\lambda$ satisfies Bohr's theorem, may be  decided within  $\lambda$-polynomials. Indeed, yet another consequence of Theorem \ref{maininftyvectorvalued}  confirms this  intuition.

\begin{Theo}Let $\lambda$ be an arbitrary frequency and $X$ a Banach space. Then Bohr's theorem holds for $\lambda$ if and only if for every $\sigma>0$ there is a constant $C=C(\sigma,\lambda)$  and $M_{0}=M_{0}(\sigma)\in \N$ such that for every $M\ge M_{0}$ and every sequence $(a_{n})\subset X$ we have
\begin{equation} \label{BTineqvectorvalued}
\sup_{N\le M} \sup_{t\in \R} \big\| \sum_{n=1}^{N}a_{n}e^{-\lambda_{n}it}\big\|_{X}\le C e^{\lambda_{M}\sigma} \sup_{t\in \R} \big\| \sum_{n=1}^{M} a_{n}e^{-\lambda_{n}it}\big\|_{X}.
\end{equation}
\end{Theo}
\begin{proof}
Assume that Bohr's theorem holds for $\lambda$. Then by Theorem \ref{maininftyvectorvalued} we know that $\mathcal{D}_{\infty}(\lambda,X)$ is complete. Hence an application of the uniform boundedness principle shows that for every $\sigma>0$ there is a constant $C=C(\sigma)>0$ such that for every $D\in \mathcal{D}_{\infty}(\lambda,X)$ we have
$$\sup_{N} \sup_{t\in \R} \big\|\sum_{n=1}^{N} a_{n}(D)e^{-\lambda_{n}(\sigma+it)}\big\|_{X} \le C(\sigma)\|D\|_{\infty}.$$
In particular, for every complex sequence $(a_{n})\subset X$
\begin{equation} \label{BTineq2vectorvalued}
\sup_{N\le M} \sup_{t\in \R} \big\| \sum_{n=1}^{N}a_{n}e^{-\lambda_{n}(\sigma+it)}\big\|_{X}\le C(\sigma) \sup_{t\in \R} \big\| \sum_{n=1}^{M} a_{n}e^{-\lambda_{n}it}\big\|_{X}.
\end{equation}
Let us now verify $(\ref{BTineqvectorvalued})$ with $M_{0}=1$. For some fixed $(a_{n})\subset X$ we define $S_{x}(s):=\sum_{\lambda_{n}<x}a_{n}e^{-s\lambda_{n}}.$
Then using Abel summation we have for every $0<y\le x$ and $t\in \R$
$$S_{y}(it)=e^{y\sigma}S_{y}(\sigma+it)-\sigma \int_{0}^{y} e^{\sigma a} S_{a}(\sigma+it) da.$$
Taking norms and applying (\ref{BTineq2vectorvalued}) we obtain
\begin{equation*}
\|S_{y}(it)\|_{X}\le e^{y\sigma} C(\sigma) \|S_{x}\|_{\infty}+ C(\sigma)\|S_{x}\|_{\infty} e^{2\sigma x}\le 2C(\sigma)e^{2\sigma x} \|S_{x}\|_{\infty},
\end{equation*}
which implies (\ref{BTineqvectorvalued}).
Assume now that (\ref{BTineqvectorvalued}) holds with a constant $C(\sigma)$, and let $D=\sum a_{n}e^{-\lambda_{n}s} \in \mathcal{D}_{\infty}^{ext}(\lambda)$.
We claim that $\sigma_{u}(D)\le 0$. First we show that (\ref{BTineqvectorvalued}) implies that for every $M\ge M_{0}(\sigma)$
\begin{equation} \label{BTineq3vectorvalued}
\sup_{N\le M} \sup_{t\in \R} \big\| \sum_{n=1}^{N}a_{n}e^{-\lambda_{n}(\sigma+it)}\big\|_{X}\le C_{1}(\sigma) \sup_{t\in \R} \big\| \sum_{n=1}^{M} a_{n}e^{-\lambda_{n}it}\big\|_{X}.
\end{equation}
Indeed, keeping the definition of $S_{x}$  again by Abel summation
$$S_{y}(2\sigma+it)=e^{-2y\sigma}S_{y}(it)+2\sigma \int_{0}^{y} e^{-2\sigma a} S_{a}(it) da$$
and so by (\ref{BTineqvectorvalued}) for every $M_{0}(\sigma)\le y \le x$
\begin{align*}
|S_{y}(2\sigma+it)| &\le \|S_{x}\|_{\infty} \big(e^{-y\sigma}C(\sigma)+2\sigma C(\sigma) \int_{0}^{y} e^{-\sigma a}da \big). \\ &\le \|S_{x}\|_{\infty}C(\sigma)\big(1+2\sigma \int_{0}^{\infty} e^{-\sigma a} da \big).
\end{align*}
From this (\ref{BTineq3vectorvalued}) follows. Now let $\sigma, \varepsilon>0$. Applying (\ref{BTineq3vectorvalued}) to the Dirichlet polynomial $R_{x}^{\lambda,1}(D_{\varepsilon})$, we obtain for every $N$ with $M_{0}(\sigma)\le \lambda_{N}\le x$
\begin{equation*}
\big\| \sum_{n=1}^{N-1} a_{n}\big(1-\frac{\lambda_{n}}{x}\big)e^{-(\sigma+\varepsilon+s)\lambda_{n}} \big\|_{\infty}=\|S_{\lambda_{N}}(R_{x}^{\lambda,1}(D_{\varepsilon+\sigma}))\|_{\infty}\le C(\sigma) \|R_{x}^{\lambda,1}(D_{\varepsilon})\|_{\infty}.
\end{equation*}
 Tending $x\to \infty$ we obtain (with Theorem \ref{uniformRieszlimitvectorvalued}) for every $N$ such that $\lambda_{N} \ge M_{0}(\sigma)$
 $$\big\|\sum_{n=1}^{N-1} a_{n}e^{-(\sigma+\varepsilon+s)\lambda_{n}}\big\|_{\infty}\le C(\sigma)\|D_{\varepsilon}\|_{\infty} \le C(\sigma)\|D\|_{\infty}.$$
We conclude by Proposition \ref{BohrCahenIIvectorvalued} that $\sigma_{u}(D)\le \sigma+\varepsilon$ for every $\sigma,\varepsilon>0$, and so we have $\sigma_{u}(D)\le 0$.
\end{proof}

\subsection{Bohr's strips}
Given a frequency $\lambda$ and a  Banach space $X$,  we define
\[
L(\lambda,X) := \sup_{D \in \mathcal{D}(\lambda,X)} \sigma_{a}(D)-\sigma_{c}(D)\,,
\]
and abbreviate $L(\lambda) = L(\lambda,\C)$. Then straightforward arguments show that
\begin{equation} \label{Lstrip1vectorvalued}
L(\lambda)= L(\lambda,X) =\sigma_{c}\left(\sum e^{-\lambda_{n}s} \right)=
\sigma_{a}\left(\sum e^{-\lambda_{n}s} \right)  \,,
\end{equation}
and (with a less obvious argument)  Bohr proved in \cite[\S 3, Hilfssatz 2 and 3]{Bohr2} that
\begin{equation} \label{Lstrip2vectorvalued}
L(\lambda) =\limsup_{N \to \infty} \frac{\log(N)}{\lambda_{N}}\,.
\end{equation}
Define also
\begin{equation*}
S(\lambda,X):=\sup_{D\in \mathcal{D}(\lambda,X)} \sigma_{a}(D)-\sigma_{u}(D)
\end{equation*}
(again we write $S(\lambda) = S(\lambda,\C)$), and note that under Bohr's theorem for $\lambda$
(see also Proposition~\ref{mainresultingredient1vectorvalued}) we have
\begin{equation} \label{Lstrip4vectorvalued}
S(\lambda,X)=\sup_{D\in \mathcal{D}_{\infty}(\lambda,X)} \sigma_{a}(D).
\end{equation}
Then in the  ordinary case
\begin{equation} \label{Lstrip3vectorvalued}
S((\log n), X)=1-\frac{1}{\text{\text{cot(X)}}}\,,
\end{equation}
where $\text{cot(X)}$ denotes the  optimal cotype  of $X$. More precisely,  for finite dimensional $X$ we have
that $$S((\log n), X)=\frac{1}{2}\,,$$
which  is a celebrated theorem of Bohnenblust and Hille
from \cite{BoHi}, and for infinite dimensional $X$ the result was proved in
\cite{DefantGarciaMaestrePerez} (see also \cite[Theorem 26.4]{Defant}).

For any frequency $\lambda$ and any finite (!) dimensional Banach space $X$, we by the Cauchy-Schwarz inequality deduce that
\begin{equation*}
S(\lambda, X) \leq \frac{L(\lambda)}{2}\,.
\end{equation*}
But this estimate is far from being an equality:
For the scalar case $X=\mathbb{C}$ a results of Neder from \cite{Neder} shows that for every $x>0$ and $0\le y\le\frac{x}{2}$ there is a frequency $\lambda$ such that
$S(\lambda,\mathbb{C})=y$ and $L(\lambda)=x$.
Regarding the case $x=\infty$, for $\mathbb{Q}$-linearly independent frequencies $\lambda$ we have $S(\lambda,\C)=0$ (see \cite[Theorem 4.7]{Schoolmann}) although $L(\lambda)=\infty$ for suitable choices of $\lambda$. Hence it seems that for  scalar-valued general Dirichlet series it probably only makes sense to ask for the exact value of $S(\lambda, \mathbb{C})$ for concrete (families of) frequencies.
The game seems to change drastically if we consider $X$-valued $\lambda$-Dirichlet series for an infinite dimensional Banach space $X$.

\begin{Prop}
	\label{lowvectorvalued}
	Let $\lambda$ be any frequency satisfying Bohr's theorem, and $X$ an infinite dimensional Banach space. Then
	\begin{equation} \label{Lstrip5vectorvalued}
	L(\lambda) \big(1-\frac{1}{\operatorname{cot}(X)}\big)\le S(\lambda,X).
	\end{equation}
	Moreover, if $\lambda$ is $\mathbb{Q}$-linearly independent, then $$L(\lambda) \big(1-\frac{1}{\operatorname{cot}(X)}\big)= S(\lambda,X).$$
\end{Prop}
\begin{proof}
	First we show that $S(\lambda,X)$ equals the infimum of all $\sigma\in \R$ for which there is a constant $C>0$ such that for every $N$ and sequence $(a_{n})\subset X$ we have
	\begin{equation} \label{20novvectorvalued}
	\sum_{n=1}^N \|a_n\|_{X}  e^{-\lambda_n \sigma}
	\leq C  \sup_{t\in \R}\big\| \sum_{n=1}^N a_n e^{-\lambda_n it} \big\|_{X}.
	\end{equation}
	Indeed, if $\sigma >S(\lambda,X)$, then a closed graph argument (here use completeness of $\mathcal{D}_\infty(\lambda, X)$ which is guaranteed by assuming Bohr's theorem and Theorem \ref{maininftyvectorvalued}) gives (\ref{20novvectorvalued}). Conversely, let us denote by $A$ the infimum above and take $\sigma>A$. Fix $\varepsilon>0$. Then by the definition of $A$ we have for every $D=\sum a_{n}e^{-\lambda_{n}s}$ that
	$$\lim_{N\to \infty} \sum_{n=1}^N \|a_n\|_{X}  e^{-\lambda_n (\sigma+\varepsilon)}
	\leq C  \lim_{N\to \infty} \sup_{t\in \R}\big\| \sum_{n=1}^N a_n e^{-\lambda_n (\varepsilon+it)} \big\|_{X}=\|D_{\varepsilon}\|_{\infty},$$
	where the last equality follows from Bohr's theorem. Hence $\sigma+\varepsilon\ge S(\lambda,X)$ and so altogether we obtain $A=S(\lambda,X)$. Let us write $\frac{1}{\operatorname{cot}(X)'}=1-\frac{1}{\operatorname{cot}(X)}$. Then a direct calculation shows
	\[
	L(\lambda) \Big(1-\frac{1}{\operatorname{cot}(X)}\Big) = L\Big(\frac{\lambda}{\operatorname{cot}(X)'}\Big)
	=
	\inf
	\big\{ \sigma \in \R \colon  (e^{-\lambda_n\sigma}) \in \ell_{\operatorname{cot}(X)'}\big\}=:B\,.
	\]
	Let us show that $B\le A$, and assume without loss of generality that $A < \infty$. Take some  $A<\sigma$. i.e. there is $C>0$ such that for every $N$ and sequence $(a_{n})\subset X$ we have
	\begin{equation} \label{eisa}
\sum_{n=1}^N \|a_n\|_{X}  e^{-\lambda_n \sigma}
	\leq C  \sup_{t\in \R}\big\| \sum_{n=1}^N a_n e^{-\lambda_n it} \big\|_{X}.
\end{equation}
	 We show that $(e^{-\lambda_n\sigma}) \in \ell_{\text{cot(X)}'}$, that is $B\le \sigma$.
	Since $X$ is infinite dimensional, there are $x_1, \ldots, x_n \in X$ such that for all $u=(u_1, \ldots , u_N) \in \mathbb{C}^{N}$ we have
	\begin{equation} \label{infinitebanachargumentvectorvalued}
	\frac{1}{2}\|u\|_\infty \leq \|\sum_{n=1}^N x_n u_n\|_X \leq \|u\|_{\text{cot(X)}}\,;
	\end{equation}
	see \cite{MauPi} and also \cite[Theorem 14.5]{DiJaTo95}. Now let $w_1, \ldots , w_n \in \mathbb{C}$ arbitrary. Then, applying (\ref{eisa}) and (\ref{infinitebanachargumentvectorvalued}) with $a_n=e_{n}w_{n}$ and $u=(w_{1}, \ldots, w_{n})$, we obtain by the choice of $\sigma$  that
	\begin{align*}
	&
	\sum_{n=1}^N |e^{-\lambda_n \sigma} w_n|
	\leq 2 \sum_{n=1}^N \|x_nw_n\|_{X}e^{-\lambda_n \sigma}\le 2 C(\sigma) \sup_{t \in \R} \big\|  \sum_{n=1}^N  x_n w_n e^{-\lambda_n it} \big\|_{ X}
	\\&
	\leq 2 C(\sigma) \big(\sum_{n=1}^N |w_n e^{-i\lambda_n t}|^{\operatorname{cot}(X)}\big)^{\frac{1}{\operatorname{cot}(X)}} = 2 C(\sigma) \big(\sum_{n=1}^N |w_n |^{\operatorname{cot}(X)}\big)^{\frac{1}{\operatorname{cot}(X)}}\,.
	\end{align*}
	Consequently, by duality $(e^{-\lambda_{n}\sigma})\in \ell_{\operatorname{cot}(X)'}$, which finally implies $B\le A$. It remains to verify that $B\ge A$, whenever $\lambda$ is $\mathbb{Q}$-linearly independent. Note that $\T^{\infty}$ with the mapping
	$$\beta\colon \R \to \T^{\infty}, ~~ t \mapsto (e^{-it\lambda_{n}})$$
	forms a $\lambda$-Dirichlet group. We assume that $B<\infty$ and take $\sigma>B$.  Then there is some $0<\varepsilon<1$ such that $\sigma> L(\lambda)\big(1-\frac{1}{\operatorname{cot}(X)+\varepsilon}\big)$. Defining $q:=\operatorname{cot}(X)+\varepsilon$ we obtain for every sequence $(a_{n})\subset X$
	\begin{align*}
	\sum_{n=1}^{N} \|a_{n}\|_{X}e^{-\lambda_{n}\sigma}&\le C(\sigma) \left(\sum_{n=1}^{N}\|a_{n}\|^{q}_{X}\right)^{\frac{1}{q}}\le C(X,\sigma) \left( \int_{\T^{\infty}} \big\| \sum_{n=1}^{N} a_{n} z_{n} \big\|^{q}_{X} dz\right)^{\frac{1}{q}} \\ &\le C(X,\sigma) \sup_{z\in \T^{\infty}} \big\|\sum_{n=1}^{N} a_{n} z_{n} \big\|_{X}=\sup_{t\in \R} \big\|\sum_{n=1}^{N} a_{n} e^{-it\lambda_{n}} \big\|_{X},
	\end{align*}
	where the first inequality follows form H\"{o}lders inequality. This implies $B\ge A$ and finishes the proof.
\end{proof}

The following proposition gives some  evidence that  the estimate in \eqref{Lstrip5vectorvalued} (as in the particular case \eqref{Lstrip3vectorvalued}) might be an equality.

\begin{Prop} \label{conjecturevectorvalued}
	For every
	frequency $\lambda$ satisfying Bohr's theorem  we have
	\begin{equation}
	S(\lambda, \ell_r)
	=
	\begin{cases}\frac{L(\lambda)}{2}, & \text{ if }  1 \leq r \leq 2, \\[2ex]
	L(\lambda) \big( 1- \frac{1}{r} \big)      , & \text{ if } 2 \leq r \leq \infty.\\
	\end{cases}
	\end{equation}
\end{Prop}

Notice that $\ell_r$ can be replaced by any $\mathfrak{L}_r$-space. We prepare the proof with some more preliminaries.  Define for $1 \leq p \leq \infty$ in analogy to \eqref{Lstrip4vectorvalued}
\begin{equation} \label{Lstrip6vectorvalued}
S_p(\lambda,X)=\sup_{D\in \mathcal{H}_{p}(\lambda,X)} \sigma_{a}(D),
\end{equation}
and again $S_p(\lambda) = S_p(\lambda,\C)$.
Using the hypercontractivity for frequencies $\lambda$ satisfying $L(\lambda) < \infty$, we know from \cite[Appendix B]{Schoolmannthesis} that for all $1 \leq p < \infty$
\begin{equation} \label{Lstriptvectorvalued}
S_p(\lambda)= \frac{L(\lambda)}{2}\,,
\end{equation}
whenever $L(\lambda) < \infty$.

In passing, we mention that it would be very interesting to know whether this equality also holds for frequencies $\lambda$
with $L(\lambda) = \infty$; note that if $S_p(\lambda) = \infty$, then $L(\lambda)= \infty$, but the converse of this implication  seems not clear.
Anyway, in view of Neder's result mentioned above, the behavior of $S_p(\lambda)$ seems not at all as chaotic as $S(\lambda)$.

We need the following alternative descriptions of $S_p(\lambda, X)$ (the proof of which follows by a standard closed graph argument and the Bohr-Cahen formulas for $\sigma_a$
and $\sigma_u$; analyse e.g. the proof of \cite[Proposition 9.5]{Defant}):
Assume that $\lambda$ satisfies Bohr's theorem, $X$ is some Banach space, and $1 \leq p \leq \infty$. Then $S_p(\lambda, X)$ equals the infimum over all $\sigma >0$
for which  there is a constant $c_\sigma >0$ such that for all finite Dirichlet polynomials $D = \sum_{\lambda_n \leq x} a_n e^{-\lambda_n s}$
\begin{equation} \label{charac1vectorvalued}
\sum_{\lambda_n \leq x} \|a_n\|e^{-\lambda_n \sigma} \leq c_\sigma \|D\|_p,
\end{equation}
and alternatively we may replace this estimate by
\begin{equation} \label{charac2vectorvalued}
\sum_{\lambda_n \leq x} \|a_n\| \leq c_\sigma e^{x \sigma}\|D\|_p .
\end{equation}

\begin{proof}[Proof of Proposition~\ref{conjecturevectorvalued}]
	By Proposition~\ref{lowvectorvalued}
	 we only have to prove the upper bound.

	\noindent The case $r = \infty$: By the definitions and  \eqref{Lstrip1vectorvalued} (first equality) we obviously have $S(\lambda, \ell_\infty) \leq L(\lambda, \ell_\infty) = L(\lambda)$, which is what we want. For the rest of this proof fix some $\lambda$-Dirichlet group $(G, \beta)$.

	\noindent The case $1 \leq r \leq 2$:  We first handle the case $r=1$. Since  $S(\lambda, \ell_1) \leq S_1(\lambda, \ell_1)\,,$
	it suffices to show that
	$$S_1(\lambda, \ell_1) \leq \frac{L(\lambda)}{2}\,.$$
	Fix  some $\varepsilon >0$ and take  a Dirichlet polynomial $D = \sum_{\lambda_n \leq x} a_n e^{-\lambda_n s} \in \mathcal{D}(\lambda, \ell_1)$.
	Then by \eqref{Lstriptvectorvalued} for $p=1$ we have
	\begin{align*}
	\sum_{\lambda_n \leq x}\| a_n\|_{\ell_1}
	&
	=\sum_n \sum_k |a_n(k)| = \sum_k \sum_n |a_n(k)|
	\\&
	\leq \sum_k c_{\varepsilon} e^{x\big(\frac{L(\lambda)}{2}+\varepsilon\big)}
	\int_G|\sum_{\lambda_n \leq x}  a_n(k) h_{\lambda_n}(\omega )| d\omega
	\\&
	=  c_{\varepsilon} e^{x\big(\frac{L(\lambda)}{2}+\varepsilon\big)}
	\int_G \sum_k|\sum_{\lambda_n \leq x}  a_n(k) h_{\lambda_n}(\omega )| d\omega
	\\&
	=  c_{\varepsilon} e^{x\big(\frac{L(\lambda)}{2}+\varepsilon\big)}
	\int_G \|\sum_{\lambda_n \leq x}  a_n h_{\lambda_n}(\omega )\|_{\ell_1} d\omega,
	\end{align*}
	and the conclusion  follows from \eqref{charac2vectorvalued}.
	Assume now  that $1 < r \leq 2. $ Then it is well-known that there is a finite measure $\mu$ and an isometric embedding
	$\ell_r \hookrightarrow L_1(\mu)$ (see e.g. \cite[Section 24]{DefantFloret}). Since  $L_1(\mu)$ is an $\mathfrak{L}_1$-space
	(see e.g. \cite[Section 23]{DefantFloret}), we by \eqref{charac1vectorvalued} or  \eqref{charac2vectorvalued}  easily deduce that
	$S(\lambda, L_1(\mu)) = S(\lambda, \ell_1)$. But then we conclude  from the case $p=1$ that as desired
	\[
	S(\lambda, \ell_r)  \leq  S(\lambda, L_1(\mu)) \leq \frac{L(\lambda)}{2}\,.
	\]
	
	\noindent
	The case $2 \leq r \leq \infty$:
	For every $D = \sum_{\lambda_n \leq x} a_n e^{-\lambda_n s} \in \mathcal{D}(\lambda, \ell_r)$ we have
	\begin{align*}
	\sum_{\lambda_n \leq x}\| a_n\|_{\ell_r} e^{-\frac{L(\lambda)+\varepsilon}{r'}\lambda_n}
	&
	\leq \Big( \sum_{\lambda_n \leq x}  e^{-\big(L(\lambda)+ \varepsilon\big)\lambda_n}  \Big)^{\frac{1}{r'}} \,\Big( \sum_{\lambda_n \leq x}\|a_n\|_{\ell_r}^r \Big)^{\frac{1}{r}}
	\\&
	\leq C \Big( \sum_k \sum_{\lambda_n \leq x}|a_n(k)|^r \Big)^{\frac{1}{r}}\,.
	\intertext{Applying the Hausdorff-Young inequality for locally compact abelian Dirichlet groups and Minkowski's integral inequality we get}
	\sum_{\lambda_n \leq x}\| a_n\|_{\ell_r} e^{-\frac{L(\lambda)+\varepsilon}{r'}\lambda_n}
	&
	\leq C \Big( \sum_k\Big( \int_G |\sum_{\lambda_n \leq x}  a_n(k) h_{\lambda_n}(\omega )|^{r'} d\omega\Big)^{\frac{r}{r'}} \Big)^{\frac{1}{r}}
	\\&
	\leq C\Big( \int_G \|\sum_{\lambda_n \leq x}  a_n h_{\lambda_n}(\omega )\|_{\ell_r}^{r'} d\omega\Big)^{\frac{1}{r'}}\,.
	\end{align*}
	This proves by \eqref{charac1vectorvalued} that
	\[
	S(\lambda, \ell_r) \leq S_{r'}(\lambda, \ell_r) \leq \frac{L(\lambda)}{r'}\,. \qedhere
	\]
\end{proof}

The statement for $2\le r < \infty$ in Proposition \ref{conjecturevectorvalued} can be generalized to spaces with type 2.

\begin{Prop} \label{type2}
	Given a frequency $\lambda$ satisfying Bohr's theorem and an infinite dimensional Banach space $X$ of type 2, we have
\begin{equation}\label{eq-bohr-strip}
  S(\lambda,X)=L(\lambda)(1-1/\text{cot}(X)).
\end{equation}
\end{Prop}

We prepare the proof with the following lemma -- an argument  that is part  of the proof of  \cite[Theorem 4.1]{CarandoMaceraScottiTradacete}.
\begin{Lemm} \label{RUCvectorvaluedA}
	Let $X$ be a Banach space of type $2$ and $(G, \beta)$ a $\lambda$-Dirichlet group. Then there is a constant $C>0$ such that for every choice of finitely many $a_1, \ldots, a_m \in X$ we have
	\begin{align}\label{RUCvectorvalued}
	\mathbb{E} \Big\|\sum_{n=1}^m \varepsilon_n  a_n \Big\|_{X}^2
	\leq C \int_G \big\|\sum_{n=1}^m  a_n h_{\lambda_n}(\omega )\big\|_{X}^{2} d\omega,
	\end{align}
	where the $\varepsilon_n$ form  independent identical distributed Bernoulli variables.
\end{Lemm}
For the sake of completeness we sketch the proof.

\begin{proof}[Proof of Lemma~\ref{RUCvectorvaluedA}] Denote $D=\sum_{n=1}^m a_n h_{\lambda_n}$ and let $T: \ell_2^m \to X$ be the operator defined by $T(e_n)=a_n$.
	Since $X$ has type $2$, we have
	\[
	\Big(\mathbb{E} \Big\|\sum_{n=1}^m \varepsilon_n  a_n \Big\|_{X}^2\Big)^{\frac{1}{2}}
	\ll
	\Big(\mathbb{E} \Big\|\sum_{n=1}^m \gamma_n  a_n \Big\|_{X}^2 \Big)^{\frac{1}{2}} \ll \pi_2(T^\ast);
	\]
	see e.g \cite[(4.2) and Theorem 12.2]{Tomczak-Jaegermann}. For every $x^*\in X^*$ observe that
	\[\|x^*(D)\|_{L_2(G)}
	=\|(x^*(a_n))_{n=1}^m\|_{\ell_2^m}
	=\|T^*(x^*)\|_{\ell_2^m}.\]
	Therefore, given  a finite collection of vectors $x_k^*\in X^*$ we have
	\begin{align*}
	\sum_k \|T^*(x_k^*)\|_{\ell_2^m}^2
	&=\sum_k \|x_k^*(D)\|_{L_2(G)}^2=\int_G \sum_k |x_k^*(D(\omega))|^2 d\omega
	\\ & \leq \int_G \|D(\omega)\|_X^2 \sup_{x^{**}\in B_{X^{**}}}\sum_k |x^{**}(x_k^*)|^2  d\omega
	\\ &= \|D\|_{L_2(G,X)}^2 \sup_{x^{**}\in B_{X^{**}}}\sum_k |x^{**}(x_k^*)|^2.
	\end{align*}
	From the definition of the 2-summing norm we deduce that $\pi_2(T^\ast)\leq \|D\|_{L_2(G,X)}$ which concludes the argument.
\end{proof}

Having Lemma \ref{RUCvectorvaluedA} in mind, in order to prove Proposition~\ref{type2} we check that  formula \eqref{eq-bohr-strip} for Bohr's strip holds whenever \eqref{RUCvectorvalued} is satisfied. Therefore, it is natural to ask for which Banach spaces and frequencies inequality \eqref{RUCvectorvalued} holds. As it turns out, by \cite[Theorem 4.1]{CarandoMaceraScottiTradacete}
we have that \eqref{RUCvectorvalued} is equivalent to type 2 for the case ordinary Dirichlet series $\lambda=(\log n)$. Also, combining \cite[Proposition 3.1]{CarandoMaceraScottiTradacete} and \cite[Theorem 4.1]{CarandoMaceraScottiTradacete} or  looking carefully at the proof of \cite[Theorem 1.5]{ArendtBu} we have also that \eqref{RUCvectorvalued} is equivalent to type 2 for the case of classical Fourier series $\lambda=(n)$. On the other hand, there are frequencies for which  \eqref{RUCvectorvalued} is satisfied for every Banach space:  $\Q$-linearly independent frequencies (trivially) and lacunary frequencies (see \cite[Theorem 2.1]{Pi77}). As a consequence, in these cases we have $S(\lambda,X)=L(\lambda)(1-1/\text{cot}(X))$ regardless of the Banach space $X$.

\begin{proof}[Proof of Proposition~\ref{type2}]
	Fix some $\lambda$-Dirichlet group $(G, \beta)$. Since $X$ has type 2, we have that $\operatorname{cot}(X)<\infty$. Then for $\text{cot}(X)<q<\infty$ and sequence $(a_{n})\subset X$ we  for each $x$ obtain
	\begin{align*}
	\sum_{\lambda_n \leq x}\| a_n\|_{X} e^{-\frac{L(\lambda)+\varepsilon}{q'}\lambda_n}
	&
	\leq \Big( \sum_{\lambda_n \leq x}  e^{-\big(L(\lambda)+ \varepsilon\big)\lambda_n}  \Big)^{\frac{1}{q'}} \,\Big( \sum_{\lambda_n \leq x}\|a_n\|_{X}^q \Big)^{\frac{1}{q}}\,.
	\intertext{Applying the cotype $q$ inequality and \eqref{RUCvectorvalued}, we obtain}
	\sum_{\lambda_n \leq x}\| a_n\|_{X} e^{-\frac{L(\lambda)+\varepsilon}{q'}\lambda_n}
	& \leq C \Big(\mathbb{E} \Big\|\sum_{\lambda_n \leq x} \varepsilon_n  a_n \Big\|_{X}^2\Big)^{\frac{1}{2}}
	\\&
	\leq C\Big( \int_G \|\sum_{\lambda_n \leq x}  a_n h_{\lambda_n}(\omega )\|_{X}^{2} d\omega\Big)^{\frac{1}{2}}\,.
	\end{align*}
		By \eqref{charac1vectorvalued} we deduce  that
		\[
		S(\lambda, X) \leq S_{2}(\lambda, X) \leq \frac{L(\lambda)+\varepsilon}{q'}\,,
		\]
		for every $\varepsilon>0$ and every $q>\operatorname{cot}(X)$, which concludes the argument.
\end{proof}

\subsection{$\mathbb{Q}$-linear independence II}\label{subsec-QlinindIIvectorvalued}
 In the case of a $\mathbb{Q}$-linearly independent frequency $\lambda$ we know from Theorem \ref{radHpcoincidevectorvalued} that the equality $\mathcal{H}_{\infty}^{+}(\lambda,X)=\mathcal{H}_{\infty}(\lambda,X)$ holds if and only if $c_{0}$ is not isomorphically contained in $X$. Moreover, for this class of $\lambda$'s and for the scalar case $X=\mathbb{C}$, in  \cite[Theorem 4.7]{Schoolmann} it is shown that isometrically
\begin{equation} \label{Qlinscalarvectorvalued}
\mathcal{D}_{\infty}(\lambda) = \ell_{1}, ~~ \sum a_{n}e^{-\lambda_{n}s}\mapsto (a_{n})\,.
\end{equation}
 Hence by Theorem \ref{maininftyvectorvalued} we have $\mathcal{D}_{\infty}(\lambda,X)=\mathcal{H}_{\infty}^{+}(\lambda,X)$ for every Banach space $X$. All together we obtain that for $\mathbb{Q}$-linearly independent frequencies $\lambda$ the equality
$\mathcal{D}_{\infty}(\lambda,X)=\mathcal{H}_{\infty}(\lambda,X)$ holds if and only if $c_0$
is no isomorphic copy of $X$.

In this section we provide another approach to this result by extending the equality $\mathcal{D}_{\infty}(\lambda)=\ell_{1}$ to its vector-valued analog. Therefore
let us denote by $\ell_{1}^{w}(X)$ the Banach space of all weak summable $X$-valued sequences with norm
$$w((a_{n}))=\sup_{x^{\ast} \in B_{X^{\ast}} } \sum_{n=1}^{\infty} |x^{\ast}(a_{n})|$$
and $\ell^{w,0}_{1}(X)$ is the (closed) subspace of $\ell_{1}^{w}(X)$ consisting of  $(a_{n}) \in \ell_{1}^{w}(X)$ such that
$$\lim_{N \to \infty} w\left((a_{n})_{n\ge N}\right) =0.$$
Recall from \cite[Theorem V.8]{Diestel} that  $\ell^{w,0}_{1}(X)=\ell^{w}_{1}(X)$ if and only if $c_{0}$ is not isomorphically contained in $X$.
\begin{Theo} \label{independentvectorvalued} If $\lambda$ is $\mathbb{Q}$-linearly independent then for every Banach space $X$ isometrically we have
\begin{equation} \label{independentchainvectorvalued}
\ell^{w,0}_{1}(X) = \Hcal_{\infty}(\lambda,X) \subset \mathcal{D}_{\infty}(\lambda,X) =\ell^{w}_{1}(X).
\end{equation}
 In particular,  $\mathcal{D}_{\infty}(\lambda,X)=\Hcal_{\infty}(\lambda,X)$  holds isometrically if and only if $c_0$
is no isomorphic copy of $X$.
\end{Theo}
\begin{proof} We start with the second equality in (\ref{independentchainvectorvalued}). Let $D \in \mathcal{D}_{\infty}(\lambda,X)$ with Dirichlet coefficients $(a_{n})$.
Then by (\ref{Qlinscalarvectorvalued}) and Theorem~\ref{HahnBanachinftyvectorvalued}
$$w((a_n)) = \sup_{x^{*} \in B_{X^{*}}} \sum_{n=1}^{\infty}|x^{*}(a_{n})|
=\sup_{x^{*} \in B_{X^{*}}}  \|x^\ast \circ D\|_\infty
=\|D\|_{\infty}<\infty\,.$$
Conversely, the same arguments give that $\ell_{1}^{w}(X) \subset \mathcal{D}_\infty(X)$.

 Now we verify the first equality in (\ref{independentchainvectorvalued}). Fix $(a_{n}) \in \ell^{w,0}_{1}(X)$ and a $\lambda$-Dirichlet group $(G,\beta)$. Then for all $M\ge N$ we have (using the $\ell_{1}$--$\ell_{\infty}$ duality)
\begin{align*}
\Big\| \sum_{n=N}^{M} a_{n} h_{\lambda_{n}}\Big\|_{\infty}& =\sup_{\omega \in G} \Big\| \sum_{n=N}^{M} a_{n} h_{\lambda_{n}}(\omega) \Big\|_{X} \le  \sup_{b \in B_{\ell_{\infty}}} \Big\| \sum_{n=N}^{M} a_{n} b_{n} \Big\|_{X} \\ &\le  \sup_{x^{\ast} \in B_{X^{\ast}}} \sum_{n=N}^{\infty} |x^{\ast}(a_{n})|   =  w((a_{n})_{n\ge N} ),
\end{align*}
which vanishes as  $N\to \infty$. This shows that  $(\sum_{n=1}^{M} a_{n} h_{\lambda_{n}})_{M}$ is a Cauchy sequence in $H_{\infty}^{\lambda}(G,X)$, and therefore there is  a limit $f \in H_{\infty}^{\lambda}(G,X)$ with $\|f\|_{\infty}\le w((a_{n}))$ and $\widehat{f}(h_{\lambda_{n}})=a_{n}$ for all $n$.
Conversely, let $f\in H_{\infty}^{\lambda}(G,X)$ and define for $\omega, \eta \in G$
$$F(\omega)=[ \eta\mapsto f(\eta\omega)].$$
Then  we straightforwardly see that $F\in  L_{1}(G,H_{\infty}^{\lambda}(G,X))$. Moreover, the Fourier coefficients are given by $\widehat{F}(h_{x})=\widehat{f}(h_{x}) h_{x}$, since
$$\widehat{F}(h_{x})(\eta)=\left(\int_{G} F(\omega) \overline{h_{x}}(\omega) d\omega \right)(\eta)=\int_{G} f(\eta\omega) \overline{h_{x}}(\omega) d\omega=\widehat{f}(h_{x}) h_{x}(\eta).$$
So $F\in H_{1}^{\lambda}(G,H_{\infty}^{\lambda}(G,X))$,
and then Theorem  \ref{CHindependencevectorvalued} (to be proved in the final section) implies
that for almost all $\omega\in G$ we have
$$ \sum_{n=1}^{\infty}\widehat{f}(h_{\lambda_{n}}) h_{\lambda_{n}} h_{\lambda_{n}}(\omega)= \sum_{n=1}^{\infty} \widehat{F}(h_{\lambda_{n}})h_{\lambda_{n}}(\omega)=F(\omega)$$
with convergence in $H_{\infty}^{\lambda}(G,X)$. Hence, using the inclusion in (\ref{independentchainvectorvalued}), that is a consequence of Corollary \ref{inclusionplusvectorvalued} and Theorem \ref{maininftyvectorvalued},  there is some $\omega\in G$ such that
\begin{align*}
&\lim_{N \to \infty} w((\widehat{f}(h_{\lambda_{n}}))_{n > N})=\lim_{N\to \infty}\big\|\sum_{n=1}^{N} \widehat{f}(h_{\lambda_{n}})
h_{\lambda_{n}}-f\big\|_{H_{\infty}^{\lambda}(G,X)}\\ &= \lim_{N\to \infty}\big\|\sum_{n=1}^{N} \widehat{f}(h_{\lambda_{n}})
 h_{\lambda_{n}}(\omega)h_{\lambda_{n}}-f(\omega~\cdot )\big\|_{H_{\infty}^{\lambda}(G,X)}
\\&
=\lim_{N\to \infty} \big\|\sum_{n=1}^{N} \widehat{F}(h_{\lambda_{n}}) h_{\lambda_{n}}(\omega)-F(\omega)\big\|_{H_{\infty}^{\lambda}(G,X)}=0,
\end{align*}
which is what we aimed for.
\end{proof}

\section{Maximal inequalities} \label{maximalinequalitiessecvectorvalued}
A fundamental question in Fourier analysis is to ask under which assumptions the Fourier series of $f\in H_{1}^{\lambda}(G,X)$, that is
\begin{equation} \label{questionvectorvalued}
f \sim \sum \widehat{f}(h_{\lambda_{n}})h_{\lambda_{n}}\,,
\end{equation}
represents $f$ by pointwise converges, or with respect to some norm.
More questions appear, if one here replaces ordinary summation by other summation methods.
For the scalar case, various results in this direction  can be found in \cite{DefantSchoolmann4} and \cite{DefantSchoolmann3}, and our aim in this final section is to study their vector-valued counterparts (and related topics).

\subsection{Carleson-Hunt theorem}
 A famous example in the direction of (\ref{questionvectorvalued}) for $X=\mathbb{C}$ is given by the Carleson-Hunt theorem, which states that
\begin{equation} \label{CHpowerseriesvectorvalued}
f(z)=\sum_{k=0}^{\infty} \widehat{f}(k)z^{k}
\end{equation}
almost everywhere on $\T$, provided that $f\in L_{p}(\T)$ and $1<p\le \infty$. As proven in \cite[Theorem 2.2]{DefantSchoolmann4} this results extends to $H_{p}^{\lambda}(G)$, $1<p\le \infty$, for arbitrary frequencies $\lambda$ and $\lambda$-Dirichlet groups $(G,\beta)$.
The techniques of the proof in \cite{DefantSchoolmann4} extend to $H_{p}^{\lambda}(G,X)$ once we assume that the $X$-valued counterpart of  (\ref{CHpowerseriesvectorvalued}) is valid.
\begin{Theo} \label{CHtheoremvectorvalued}
Assume that $X$ is a Banach space for which the Carleson-Hunt maximal inequality holds, i.e. for every $1 < p < \infty$ there
is some $C >0$ such that for every $f \in L_p(\mathbb{T},X)$
\begin{equation} \label{CHcrucialXvectorvalued}
  \Big\| \sup_N \big\| \sum_{n=-N}^N \widehat{f}(n)  z^n  \big\|_{X} \Big\|_{L_p(\mathbb{T})}
  \leq C \|f\|_{L_p(\mathbb{T},X)}\,.
 \end{equation}
 Then for every frequency  $\lambda$,
 every $\lambda$-Dirichlet group $(G, \beta)$,  and every
  $1 < p < \infty$ we have that for every $f \in H_p^\lambda(G,X)$
  \begin{equation} \label{CHconvergencevectorvalued}
  f = \sum_{n=1}^\infty  \widehat{f}(h_{\lambda_n})  h_{\lambda_n}
  \end{equation}
  almost everywhere on $G$, and
   \begin{equation} \label{CHinequvectorvalued}
  \Big\| \sup_N \big\| \sum_{n=1}^N \widehat{f}(h_{\lambda_n})  h_{\lambda_n}  \big\|_{X} \Big\|_{p}
 \leq C \|f\|_{p}\,,
 \end{equation}
where $C = C(p)$ is a constant which only depends on $p$.
\end{Theo}

A Banach space for which (\ref{CHcrucialXvectorvalued}) holds must have UMD (see for example  \cite{LaHy}, where also some sufficient conditions are presented). To our knowledge, the class of Banach space for which (\ref{CHcrucialXvectorvalued}) holds is not known. In the case of Banach lattices, (\ref{CHcrucialXvectorvalued}) holds true if and only if $X$ has UMD \cite{Francia}. The next result shows that for $\mathbb{Q}$-linearly independent frequencies (\ref{CHinequvectorvalued}) is valid for every Banach space $X$ and $1\le p<\infty$.
\begin{Theo} \label{CHindependencevectorvalued}
Let $\lambda$ be a $\Q$-linearly independent frequency, $(G,\beta)$ a $\lambda$-Dirichlet group, and $X$ a Banach space. Then for every $1\le p< \infty$ we have for all $f\in H_{p}^{\lambda}(G,X)$
\begin{equation}
  \Big\| \sup_N \big\| \sum_{n=1}^N \widehat{f}(h_{\lambda_n})  h_{\lambda_n}  \big\|_{X} \Big\|_{p}
 \leq 2 \|f\|_{p}.
 \end{equation}
 In particular,  we almost everywhere on $G$ have
 $$f=\sum_{n=1}^{\infty} \widehat{f}(h_{\lambda_{n}}) h_{\lambda_{n}}.$$
\end{Theo}
\begin{proof}
Consider the $\lambda$-Dirichlet group induced by $\beta: \R \rightarrow \T^\infty$ given by $\beta(t)=(e^{-\lambda_n i t})$.
We have to prove that for every $f\in H_p^\lambda(\T^\infty,X)$ we have
\[\Big(\int_{\T^\infty} \sup_{N}\Big\| \sum_{k=1}^N \widehat{f}(e_k)z_k\Big\|_X^p \, dz\big)^{1/p}
\le 2 \|f\|_p, \]
This estimate follows from Levy's inequality for Banach spaces (see for example \cite[Proposition~1.1.1]{KW}). This result is usually stated for Bernoulli random variables but in fact it holds for Steinhaus variables $\T^\infty \to \T, z \mapsto z_k$ with the same proof. For $1\le N\le M$, we apply Levy's inequality to get
\[\Big(\int_{\T^\infty} \sup_{1\le N\le M}\Big\| \sum_{k=1}^N \widehat{f}(e_k)z_k\Big\|_X^p \, dz\Big)^{1/p}\leq 2\Big(\int_{\T^\infty} \Big\| \sum_{k=1}^M \widehat{f}(e_k)z_k\Big\|_X^p\, dz\Big)^{1/p}
\le 2 \|f\|_p, \]
where in the last inequality we used Theorem \ref{Nabschnittvectorvalued} and Corollary~\ref{inclusionplusvectorvalued}. The result then follows from  the monotone convergence theorem taking $M\to \infty$.
\end{proof}

Clearly, Theorem~\ref{CHtheoremvectorvalued} does not apply to the case $p=1$.
But under Landau's condition $(LC)$,  this loss  can be offset if we replace  the function
$f \in H_1^\lambda(G,X)$ by its convolution $f \ast p_\sigma$ with the Poisson measure --
this even works for every $X$.

\begin{Theo} \label{Hel2}
Suppose, that $\lambda$ satisfies $(LC)$, $(G,\beta)$ is a $\lambda$-Dirichlet group, and $X$ a Banach space. Then for every $\varepsilon>0$ there is a constant $C=C(\varepsilon, \lambda)$ such that for all $f\in H_{1}^{\lambda}(G,X)$ we have
\begin{equation*}
\Big\|\sup_{\sigma \ge \varepsilon}\sup_{N} \big| \sum_{n=1}^{N} \widehat{f}(h_{\lambda_{n}})e^{-\sigma\lambda_{n}} h_{\lambda_{n}}(\cdot)\big| \Big\|_{1,\infty} \le C \|f\|_{1},
\end{equation*}
where $\|\cdot\|_{1,\infty}$ denotes the norm of the weak $L_{1}$-space $L_{1,\infty}(G,X)$. Moreover, for every $f\in H_{1}^{\lambda}(G,X)$ there is a null set $N\subset G$ such that for every $\sigma>0$ and every $\omega\in G-N$
\begin{equation*} \label{introdis1vectorvalued}
f*p_{\sigma}(\omega)=\sum_{n=1}^{\infty} \widehat{f}(h_{\lambda_{n}}) e^{-\sigma\lambda_{n}} h_{\lambda_{n}}(\omega).
\end{equation*}
\end{Theo}
The scalar valued variant of this result can be found in \cite[Theorem 3.2]{DefantSchoolmann4}, and its proof extends word by word to $H_{1}^{\lambda}(G,X)$.

\subsection{Almost everywhere convergence of Riesz means}
It is known that (\ref{CHconvergencevectorvalued}) may fail for $p=1$
as it does for the power series case $\lambda=(n)$. For the scalar case $X=\mathbb{C}$, a substitute for this failure is given by \cite[Theorem 2.1]{DefantSchoolmann3}, which states that, given any $k>0$, for every $f\in H_{1}^{\lambda}(G)$ we almost everywhere on $G$  have
\begin{equation} \label{Rieszmeansconvvectorvalued}
f(\omega)=\lim_{x\to \infty}\sum_{\lambda_{n}<x} \widehat{f}(h_{\lambda_{n}})(\omega) \big(1-\frac{\lambda_{n}}{x}\big)^{k} h_{\lambda_{n}}(\omega);
\end{equation}
 in the language of \cite{DefantSchoolmann3} this means that  $f$   is summable by its first $(\lambda,k)$-Riesz means almost everywhere on $G$.

Analyzing the proof of this result (more precisely, replacing there the  absolute values by the norms), we see that \eqref{Rieszmeansconvvectorvalued} extends to  $X$-valued functions without any restrictions on $X$
(in contrast to Theorem \ref{CHtheoremvectorvalued}).

\begin{Theo} \label{Hel1}
Let $\lambda$ be a frequency, $k>0$, $(G,\beta)$ any $\lambda$-Dirichlet group, and $X$ a Banach space. Then
$$R_{\max}^{\lambda,k}(f)(\omega)=\sup_{x>0} \Big\| \sum_{\lambda_{n}<x} \widehat{f}(h_{\lambda_{n}}) \big(1-\frac{\lambda_{n}}{x}\big)^{k} h_{\lambda_{n}}(\omega)\Big\|_{X}$$
defines a bounded operator from $H_{1}^{\lambda}(G,X)$ to $L_{1,\infty}(G,X)$. In particular, given $f\in H_{1}^{\lambda}(G,X)$, for every $k>0$ almost everywhere on $G$
\begin{equation*}
f=\lim_{x \to \infty}  \sum_{\lambda_{n}<x} \widehat{f}(h_{\lambda_{n}}) \big(1-\frac{\lambda_{n}}{x}\big)^{k} h_{\lambda_{n}}.
\end{equation*}
\end{Theo}

\subsection{Helson type theorems}

Let us transfer the Theorems~\ref{CHtheoremvectorvalued},~\ref{Hel2}, and~\ref{Hel1}  on
 pointwise summability of the Fourier series of functions $f\in H_{p}^{\lambda}(G,X)$ into convergence theorems of so-called vertical limits
\[
D^\omega = \sum \widehat{f}(h_{\lambda_n}) h_{\lambda_n}(\omega) e^{-\lambda_n s},\, \omega \in G
\]
of the corresponding Dirichlet series $D=\Bcal(f)\in  \mathcal{H}_{p}(\lambda,X)$.
The key to translate these results into terms of Dirichlet series is given by the  vector-valued analogs of
\cite[Remark 1.3]{DefantSchoolmann2} and \cite[Lemma 1.4]{DefantSchoolmann3}.

\begin{Theo} Let $(G, \beta)$ be a $\lambda$-Dirichlet group and $D=\sum a_{n}e^{-\lambda_{n}s}\in \mathcal{H}_{p}(\lambda,X)$, where $1\le p \le \infty$ and $X$ is a Banach space.
\begin{enumerate}
\item[(1)] Provided $1 < p < \infty$ and $X$ satisfies \eqref{CHcrucialXvectorvalued},  almost all vertical limits  $D^{\omega}, \omega \in G$, converge almost everywhere on $[Re=0]$ and so consequently at every $s\in [Re>0]$.
\item[(2)] If $\lambda$ satisfies $(LC)$ and $p=1$, then $D^{\omega}$ converges on $[Re>0]$ for almost every $\omega\in G$.
    \item[(3)] If $p=1$, then almost all $D^{\omega}, \omega \in G$ are Riesz summable almost everywhere on $[Re=0]$, that is the limit
$$\lim_{x\to \infty} \sum_{\lambda_{n}<x} a_{n}h_{\lambda_{n}}(\omega) \big(1-\frac{\lambda_{n}}{x}\big)^{k} e^{-it\lambda_{n}}$$
exists for almost all $t\in \R$. In particular, we have $\sigma_{c}^{\lambda,k}(D^{\omega})\le 0$ for almost every $\omega$.
\end{enumerate}
\end{Theo}

\subsection{Riesz projection}

We finally study the boundedness of the vector-valued Riesz projection for Dirichlet groups, or equivalently the boundedness of the vector-valued Hilbert transform in these groups.

Let $G$ be a compact abelian group and $P\subset\widehat{G}$ such that $P+P\subset P$, $P\cup(-P)=\widehat{G}$ and $P\cap(-P)=\{0\}$. Notice that $P$ (which stands for positive) defines an order on $\widehat{G}$. In the case of Dirichlet groups we will always consider the order inherited by $\R$ given by $P=\{h_x\in\widehat{G}:\ x\ge 0\}$. A distinction between positive and negative characters allows us to define a Hilbert transform, also known as abstract conjugate function. Indeed, define the Hilbert transform $T_P$ and the Riesz projection $R_P$ over $X$-valued trigonometric polynomials on $G$ by
\[T_P\Big(\sum_{\gamma\in\widehat{G}} x_\gamma \gamma\Big)= -i\sum_{\gamma\in\widehat{G}} \text{sg}(\gamma) x_\gamma \gamma \quad \text{and}\quad R_P\Big(\sum_{\gamma\in\widehat{G}} x_\gamma \gamma\Big)= \sum_{\gamma\in P}  x_\gamma \gamma.\]
where $\text{sg}(\gamma)=\chi_P(\gamma)-\chi_{-P}(\gamma)$.
A Banach space has the ACF (abstract conjugate function) property if there is $1<p<\infty$ and a constant $C>0$ such that for every compact abelian group $G$ with ordered characters, $T_P$ extends to a bounded operator on $L_p(G,X)$ with norm bounded by $C$.

As it turns out, UMD is equivalent to ACF for some (every) $1<p<\infty$ (see \cite{BeGiMu,Bou83,burk} and also \cite{AsmarKellyMonty} for a complete picture and an alternative proof of ACF $\Rightarrow$ UMD).

\begin{Theo}\label{rieszprojvectorvalued} Let $(G,\beta)$ be a Dirichlet group ($G\neq 0$) and $X$  a Banach space. Then $X$ has UMD if and only if the Riesz projection
$$R\colon L_{p}(G, X)\to H_{p}(G,X)$$ is bounded for some (and then for all) $1<p<\infty$, where we denote by $H_{p}(G,X)$ the space of all $f\in L_{p}(G,X)$ such that $\widehat{f}(h_{x})\ne 0$ implies $x\ge 0$ for every $x$.
\end{Theo}

\begin{proof}
If $X$ has UMD, then joining \cite{burk} and \cite[Theorem~2.1]{AsmarKellyMonty} we deduce that $X$ enjoys the ACF property for every $1<p<\infty$. In our setting we have $P=\{h_x\in\widehat{G}:\ x\ge 0\}$ and so $T_P$ is bounded. Therefore, the Riesz projection $R$ is bounded, since
	\[Rf=1/2(f+iT_Pf+\widehat{f}(0)),\]
	for every $f\in L_p(G,X)$.

Conversely, since $(G,\beta)$ is a Dirichlet group we have that $\widehat{G}$ is a subgroup of $\R$. Take any non-zero character $h_{x}\in \widehat{G}$ with $x>0$. Then the mapping
$$\alpha\colon \mathbb{Z} \hookrightarrow \widehat{G}, ~~ k \mapsto h_{kx},$$
where $kx$ is multiplication in $\R$, is an injective homomorphism. So the dual map of $\alpha$, that is $\widehat{\alpha} \colon G \to \T$, is continuous and has dense range. Then
$$L_{p}(\T,X) = H_{p}^{\alpha(\Z)}(G,X), ~~ f\mapsto f\circ \widehat{\alpha},$$
is an onto isometry and $\widehat{f\circ \widehat{\alpha}} \circ \alpha=\widehat{f},$ that is $\widehat{f}(k)=\widehat{f\circ \widehat{\alpha}}(h_{kx})$ for all $k\in \Z$.
(see \cite[Proposition 3.17 with $E=\Z$]{DefantSchoolmann2}, where the proof extends to the vector valued setting).  On the other hand, with $E=\mathbb{N}_{0}$ we isometrically have
\begin{equation} \label{3janvectorvalued}
H_{p}(\T,X)=H_{p}^{\alpha(\mathbb{N}_{0})}(G,X).
\end{equation} Now we apply the Riesz projection of $G$. Let $f \in L_{p}(\T,X)$ and define $g:=R(f\circ \widehat{\alpha})$. Then $\widehat{g}(h_{xk})=0$, if $k<0$, and $\widehat{g}(h_{xk})=\widehat{f}(k)$, if $k\ge 0$, since $x>0$. Hence $g \in H_{p}^{\alpha(\N_{0})}(G,X)$ and so there is a corresponding $\widetilde{f}\in H_{p}(\T,X)$ in the sense of (\ref{3janvectorvalued}). In this way we obtain a bounded map $$L_{p}(\T,X)\to  H_{p}(\T,X),~~ f \mapsto \widetilde{f},$$
that coincides with the Riesz projection on $\T$. Therefore $X$ has UMD (see e.g. \cite[Corollary 5.2.11, p. 399]{HytWeisI}).
\end{proof}

 The following corollary complements Theorem~\ref{quantitativevectorvalued}
and its consequences in the reflexive case $1<p<\infty$.

\begin{Coro} \label{projvectorvalued} Let $\lambda$ be a frequency.
If $X$ has UMD, then for all $1<p<\infty$
$$\sup_{N\in \N}\Big\|\pi^{N}_{p} \colon \mathcal{H}_{p}(\lambda,X) \to \mathcal{H}_{p}(\lambda,X), ~~ \sum a_{n}e^{-\lambda_{n}s} \mapsto \sum_{n=1}^{N} a_{n}e^{-\lambda_{n}s}\Big\| <\infty.$$
In particular,  every $\lambda$-Dirichlet series from $\mathcal{H}_{p}(\lambda,X)$ converges in $\mathcal{H}_p(\lambda,X)$.
\end{Coro}
\begin{proof}
    Let $(G,\beta)$ be a $\lambda$-Dirichlet group and identify $\mathcal{H}_{p}(\lambda,X)=H^\lambda_{p}(G,X)$. Then fixing $D\in \mathcal{H}_{p}(\lambda,X)$ we for every $N\in \N$ have
    \[\pi^{N}_{p}(D)=D -e^{-\lambda_{N+1}s}R(e^{\lambda_{N+1}s}D),\]
    that leads to
    \[\|\pi^{N}_{p}(D)\|_p\leq (1+\|R\|)\|D\|_p.\]
    Now taking supremum in $N$ we obtain by Theorem \ref{rieszprojvectorvalued} that $$\sup_N\|\pi^{N}_{p}\|\leq 1+\|R\|<\infty.$$
In order to prove the second statement fix again $D\in\mathcal{H}_{p}(\lambda,X)$. For $\varepsilon>0$ let
    $P\in\mathcal{H}_{p}(\lambda,X)$ be a Dirichlet polynomial such that $\|D-P\|_p<\varepsilon$. Choose $N_0$ such that $\pi^{N_0}_{p}(P)=P$. Then for every $N\geq N_0$ we have
    \begin{align*}
        \|D-\pi^{N}_{p}(D)\|_p &\leq \|D-P\|_p + \|P-\pi^{N}_{p}(D)\|_p
        \\ &
        \leq \varepsilon + \|\pi^{N}_{p}(D-P)\|_p
        \leq (2+\|R\|)\varepsilon.
    \end{align*}
    Hence, the partial sums of $D$ converge in $\mathcal{H}_{p}(\lambda,X)$.
\end{proof}

\end{document}